\documentclass[11pt]{article}
\usepackage{amsmath, amssymb, latexsym, times, scalerel}
\usepackage{amsthm}
\usepackage{rsfso}

\usepackage{amscd}
\usepackage{booktabs} 
\usepackage{wrapfig}
\usepackage{tikz}
\usepackage{pgfplots}
\usepackage{float}
\pgfplotsset{compat=newest}
\usepackage{authblk}
\usepackage{mathrsfs}
\usepackage{bbm}
\usepackage{graphicx}
\usepackage{stmaryrd}
\usepackage[all]{xy}
\usepackage{enumitem}
\usepackage{tikz-cd}

\usepackage{tocloft}
\let\OLDthebibliography\thebibliography
\renewcommand\thebibliography[1]{
  \OLDthebibliography{#1}
  \setlength{\parskip}{4pt}
  \setlength{\itemsep}{2pt plus 0.3ex}
}

\usepackage{hyperref}
\hypersetup{
    colorlinks=true,
    linkcolor=blue,
    filecolor=magenta,      
    urlcolor=cyan,
    citecolor=magenta
}

\usepackage{setspace}
\newtheorem{innercustomgeneric}{\customgenericname}
\providecommand{\customgenericname}{}
\newcommand{\newcustomtheorem}[2]{%
  \newenvironment{#1}[1]
  {%
   \renewcommand\customgenericname{#2}%
   \renewcommand\theinnercustomgeneric{##1}%
   \innercustomgeneric
  }
  {\endinnercustomgeneric}
}

\newcustomtheorem{customthm}{Theorem}
\newcustomtheorem{customconj}{Conjecture}  
\newcustomtheorem{customprop}{Proposition}  
\newcustomtheorem{customcor}{Corollary}

\theoremstyle{plain}
\newtheorem{thm}{Theorem}[section]
\newtheorem{lem}[thm]{Lemma}
\newtheorem{prop}[thm]{Proposition}

\newtheorem{thm-def}[thm]{Theorem-Definition}

\theoremstyle{definition}
\newtheorem{defn}[thm]{Definition}
\newtheorem{cor-def}[thm]{Corollary-Definition}

\theoremstyle{remark}
\newtheorem{rem}{Remark}

\numberwithin{equation}{section}

\makeatletter
\renewcommand{\paragraph}{%
  \@startsection{paragraph}{4}%
  {\z@}{1.25ex \@plus 1ex \@minus .2ex}{-1em}%
  {\normalfont\normalsize\bfseries}%
}
\makeatother

\usepackage{adjustbox}
\newcommand{\BigWedge}{\mathord{\adjustbox{valign=c,totalheight=0.7\baselineskip}{$\bigwedge$}}}

\newcommand{\gl}{\mathrm{GL}}
\newcommand{\ggl}{\mathfrak{gl}}

\newcommand{\bbg}{\mathbb{G}}

\newcommand{\cf}{\mathcal{F}}
\newcommand{\co}{\mathcal{O}}
\newcommand{\cp}{\mathcal{P}}
\newcommand{\cl}{\mathcal{L}}

\newcommand{\ch}{\mathcal{H}}

\newcommand{\cd}{\mathcal{D}}
\newcommand{\ce}{\mathcal{E}}
\newcommand{\cg}{\mathcal{G}}
\newcommand{\cs}{\mathcal{S}}

\newcommand{\ci}{\mathcal{I}}

\newcommand{\xx}{\mathcal{X}}

\newcommand{\ka}{\mathfrak{a}}
\newcommand{\kg}{\mathfrak{g}}
\newcommand{\kn}{\mathfrak{n}}

\newcommand{\km}{\mathfrak{m}}

\newcommand{\kss}{\mathfrak{S}}
\newcommand{\kxx}{\mathfrak{X}}
\newcommand{\kjj}{\mathfrak{J}}

\newcommand{\bn}{\mathbf{N}}
\newcommand{\bz}{\mathbf{Z}}
\newcommand{\br}{\mathbf{R}}
\newcommand{\bp}{\mathbf{P}}
\newcommand{\bq}{\mathbf{Q}}

\newcommand{\bc}{\mathbf{C}}

\newcommand{\ba}{\mathbf{A}}
\newcommand{\bs}{\mathbf{S}}

\newcommand{\ff}{\boldsymbol{f}}
\newcommand{\jj}{\boldsymbol{j}}

\newcommand{\ppi}{\boldsymbol{\pi}}
\newcommand{\pvarpi}{\boldsymbol{\varpi}}

\newcommand{\rp}{\mathrm{p}}

\newcommand{\rT}{\mathrm{T}}
\newcommand{\rS}{\mathrm{S}}

\newcommand{\val}{\mathrm{val}}
\newcommand{\lie}{\mathrm{Lie}}
\newcommand{\Ad}{\mathrm{Ad}}

\newcommand{\diag}{\mathrm{diag}}

\newcommand{\reg}{\mathrm{reg}}
\newcommand{\spec}{\mathrm{Spec}}

\newcommand{\pic}{\mathrm{Pic}}

\newcommand{\spf}{\mathrm{Spf}}

\newcommand{\Hom}{\mathrm{Hom}}
\newcommand{\ec}{\mathrm{Ec}}
\newcommand{\ind}{\mathrm{ind}}

\newcommand{\End}{\mathrm{End}}

\newcommand{\rk}{\mathrm{rk}}

\newcommand{\ep}{\epsilon}
\newcommand{\varep}{\varepsilon}

\newcommand{\sm}{\mathrm{sm}} 
 
\newcommand{\red}{\mathrm{red}} 
\newcommand{\defm}{\mathrm{Def}}

\newcommand{\jac}{\mathrm{Jac}}
\newcommand{\ql}{\mathbf{Q}_{\ell}}

\newcommand{\fq}{\mathbf{F}_{q}}

\newcommand{\cm}{{\mathcal{M}}}
\newcommand{\ct}{\mathcal{T}}

\newcommand{\cc}{\mathcal{C}}

\def\cb{{\mathcal B}}

\def\ka{{\mathfrak{a}}}

\def\kg{{\mathfrak{g}}}

\def\km{{\mathfrak{m}}}
\def\kn{{\mathfrak{n}}}

\def\kbb{{\mathfrak{B}}}

\def\kxx{{\mathfrak{X}}}

\usepackage{xcolor}
\usepackage{version}

\excludeversion{NB}

\newcommand{\pair}[1]{\langle #1\rangle}

\newcommand{\bsm}{\begin{smallmatrix}}
\newcommand{\esm}{\end{smallmatrix}}

\def\<{\langle\,}
\def\>{\,\rangle}

\def\kss{\mathfrak S}

\def\kg{\mathfrak g}
\def\kn{\mathfrak n}
\def\<{\langle}
\def\>{\rangle}
\def\={\equiv}
\def\le{\leqslant}
\def\ge{\geqslant}

\def\mod{\;{\rm mod}\;}

\def\End{{\rm End}}
\def\Hom{{\rm Hom}}

\def\fq{\mathbb{F}_{q}}

\def\val{{\rm val}}

\def\im{{\rm Im}}

\def\Ker{\mathrm{Ker\,}}

\author{Zongbin \textsc{Chen}}

\title{\bf On the cohomological purity of the affine Springer fibers}
\date{}

\begin{document}
\maketitle

\begin{abstract}

We address questions posed by G\'erard Laumon and Jean-Loup Waldspurger concerning the cohomological purity of affine Springer fibers.
More precisely, we show that an affine Springer fiber is cohomologically pure if and only if its $\xi$-stable quotient is cohomologically pure, and that this is further equivalent to the cohomological purity of a certain sequence of truncated affine Springer fibers.
We deduce from this a sheaf-theoretic reformulation of cohomological purity for affine Springer fibers.
We then compare this new criterion with a previously known one via a microlocal analysis of the relevant intersection complexes.
As a corollary, we show that both the primitive part of the cohomology of an affine Springer fiber and the cohomology of its $\xi$-stable quotient depend only on the root valuation datum of the defining element.
\end{abstract}

\tableofcontents

\section{Introduction}

Let $k$ be a finite field $\fq$ or the field $\bc$ of complex numbers, let $F=k(\!(\varep)\!)$ be the field of Laurent series over $k$, $\co=k[\![\varep]\!]$ its ring of integers.
We fix an algebraic closure $\overline{F}$ of $F$, let $\val:\overline{F}\to \bq$ be the valuation normalized by the condition $\val(\varep)=1$.

Let $G=\gl_{d}$ and let $\kg$ be its Lie algebra. 
Let $\gamma\in \kg[\![\varep]\!]$ be a regular semisimple element. Recall that the \emph{affine Springer fiber} at $\gamma$ is the closed sub ind-$k$-scheme of the \emph{affine grassmannian} $\xx=G(\!(\varep)\!)/G[\![\varep]\!]$ defined by
\begin{align*}
\xx_{\gamma}&=\big\{g\in \xx\,\big|\,\Ad(g^{-1})\gamma \in \kg[\![\varep]\!]\big\}
\\
&=\big\{L\subset F^{d} \text{ a lattice}\mid \gamma L\subset L \big\}.
\end{align*} 
They have been introduced by Kazhdan and Lusztig \cite{kl} as an affine analogue of the Springer fibers, and were used by Goresky, Kottwitz and Macpherson \cite{gkm homology} to prove the  Langlands-Shelstad fundamental lemma in the unramified case--assuming the \emph{purity hypothesis}, which states that the cohomology of $\xx_{\gamma}$ is pure in the sense of Grothendieck-Deligne.

The purity hypothesis has been the focus of my research. During my Ph. D thesis under the supervision of G\'erard Laumon, he asked two questions concerning this hypothesis.
The first one concerns the $\xi$-stable quotient of the affine Springer fibers. In their work on the weighted fundamental lemma \cite{laumon lemme 1}, Chaudouard and Laumon introduced a notion of $\xi$-stability on the Hitchin space. 
At his suggestion, we introduced a similar notion on $\xx_{\gamma}$ (\cite{chen xi} if $\gamma$ is split, and \cite{chen truncate} in general). 
Let $\xx_{\gamma}^{0}$ be the central connected component of $\xx_{\gamma}$. Then the $\xi$-stable quotient $\xx_{\gamma}^{0,\xi}/\rS$ turns out to be proper over $k$, where $\rS$ is the maximal $F$-split subtorus of the centralizer $\rT$ of $\gamma$ in $G$. 
Computations in examples suggest that $\xx_{\gamma}^{0, \xi}/\rS$ may be cohomologically pure, and Laumon asked whether this holds in general. 

The second question concerns a specific sequence of truncated affine Springer fibers. Consider the truncated affine Springer fiber
$$
\xx_{\gamma, n}=\big\{L\subset F^{d} \text{ a lattice}\mid \gamma L\subset L \text{ and } L\subset \co^{d} \text{ of index }n\big\}
$$
For the group $\gl_{4}$ and the regular diagonal matrices $\gamma\in \ggl_{4}(F)$, I  constructed affine pavings for $\xx_{\gamma, n}$ in my thesis \cite{chen gl4}, and deduced from it the purity hypothesis in this case. 
Laumon asked whether $\xx_{\gamma, n}$ always admits an affine paving for the regular diagonal matrices in $\ggl_{d}(F)$.
My experience with $\gl_{5}$ suggests that this may be too optimistic. 
At the defense of my thesis, Prof. Waldspurger asked a similar question, but with the truncation modified as
$$
\xx_{\gamma}(n)=\big\{L\subset F^{d} \text{ a lattice}\mid \gamma L\subset L \text{ and } L\subset \co[\gamma] \text{ of index }n\big\},
$$
where we have identified $F^{d}=\co[\gamma]\otimes_{\co} F$. 
Our first main result answers their questions:
Let $M_{0}$ be the centralizer of $\rS$ in $G$, it is the minimal Levi subgroup of $G$ containing $\gamma$, let $\cl(M_{0})$ be the set of Levi subgroups of $G$ containing $M_{0}$.

\begin{customthm}{1}\label{quotient main}

Let $\gamma \in \kg(F)$ be a regular semisimple element.
Suppose that $\xx_{\gamma}^{M}$ is cohomologically pure for every proper Levi subgroup $M \in \cl(M_{0})$, then $\xx_{\gamma}$ is cohomologically pure if and only if the quotient $\xx_{\gamma}^{0,\,\xi}/\rS$ is. Moreover, this is equivalent to the cohomological purity of all the truncations $\xx_{\gamma}(n)$, $n \in \bn$.

\end{customthm}

Theorem \ref{quotient main} is proved by deformation. We construct a reduced but non-irreducible spectral curve $X_{\gamma}$, and consider the relative $\xi$-stable compactified Jacobian $\ff^{\xi}: \overline{\kjj}^{\,\xi}\to \kbb$ and the relative Hilbert scheme $\ppi^{[n]}:\kxx^{[n]}\to \kbb$ attached to a miniversal deformation $\ppi:\kxx\to \kbb$ of $X_{\gamma}$.
The decompositions of the complexes $R\ff^{\xi}_{*}\ql$ and $R\ppi^{[n]}_{*}\ql$ as direct sum of intersection complexes have been worked out by Chaudouard-Laumon \cite{laumon lemme 2} and Migliorini-Shende-Viviani \cite{miglio 2} respectively.
We prove theorem \ref{quotient main} by a careful analysis of the decomposition theorems and the geometry of the $\xi$-stable compactified Jacobian $\overline{J}_{X_{\gamma}}^{\,\xi}$ and the punctural Hilbert scheme $X_{\gamma}^{[n]}$ of $X_{\gamma}$.

This work goes beyond merely drawing inferences from established results or being driven by pure nostalgia; rather, it stems from an effort to gain a deeper understanding of our work \cite{chen decomposition}.
In this work, we propose an alternative approach to \emph{deforming} the affine Springer fibers. 
Let $C_{\gamma}$ be an irreducible rational projective curve with unique singularity isomorphic to $\spec(\co[\gamma])$. Let $\pi:(\cc, C_{\gamma})\to (\cb, 0)$ be an algebraization of a miniversal deformation of $C_{\gamma}$, let $f:\overline{\cp}\to \cb$ be the relative compactified Jacobian of $\pi$. 
The affine Springer fiber $\xx_{\gamma}$ is homeomorphic to the universal abelian covering of the compactified Jacobian $\overline{P}_{C_{\gamma}}=\overline{\pic}{}^{0}_{C_{\gamma}/k}$. 
We consider finite abelian coverings $\overline{P}_{n}$ of $\overline{P}_{C_{\gamma}}$, and their miniversal equivariant deformations $f_{n}:\overline{\cp}_{n}\to \cb_{n}$, $n\in \bn, (n, p)=1$. They form a projective system and can be regarded as a substitute for the deformation of $\xx_{\gamma}$. 
Let $j:\cb^{\sm}\to \cb$ be the inclusion of the open subscheme over which $\pi$ is smooth, let $\pi^{\sm}$ be the restriction of $\pi$ to the inverse image of $\cb^{\sm}$. 
Let $\cf=R^{1}\pi^{\sm}_{*}\ql$ and $\ce\subset \cf|_{U}$ the local system of vanishing cycles for the singularity of $C_{\gamma}$, where $U$ is a sufficiently small open neighbourhood of $0$, let $\ce^{\perp}$ be the orthogonal complement of $\ce$ with respect to the cup product of $\cf$. 
We obtain a decomposition of $Rf_{n, *}\ql$ as a direct sum of intersection complexes, and deduce from it that 
\begin{align*}
H^{i}\big({\xx}_{\gamma}^{0},\bq\big)=&
\bigoplus_{i'=0}^{i} 
{\rm Im}\Big\{\ch^{i-i'}\big(j_{!*}\BigWedge^{i'}\cf\big)_{0}\to  
\ch^{i-i'}\big(j_{!*}\BigWedge^{i'}(\cf/\ce^{\perp})\big)_{0} \Big\}\oplus \\
&
\bigoplus\text{ Similar terms for the Levi subgroups}. 
\end{align*}
The cohomological purity of $\xx_{\gamma}$ is then equivalent to that of the \emph{primitive part}
\begin{equation}\label{springer sheaf}\tag{1}
{\rm Im}\Big\{\ch^{*}\big(j_{!*}\BigWedge^{i}\cf\big)_{0}\to  
\ch^{*}\big(j_{!*}\BigWedge^{i}(\cf/\ce^{\perp})\big)_{0} \Big\}.
\end{equation}
Thus, we obtain a sheaf-theoretic reformulation of the purity hypothesis, which provides a framework to attack the hypothesis from the perspective of deformations.

But there is a technical difficulty in this work: the sheaf $\cf/\ce^{\perp}$ lives on a sufficiently small open neighbourhood of $0$ and can not be globalized. 
Meanwhile, its subsheaf $\ce/\ce^{\perp}$ is irreducible and has a chance of being globalized. 
We have speculated that the image of $(\ref{springer sheaf})$ actually lies in the image of $\ch^{*}\big(j_{!*}\BigWedge^{i}(\ce/\ce^{\perp})\big)_{0}$, and that the cohomological purity of $\xx_{\gamma}$ should follow from the purity of this object.
It is the realization that $\ce/\ce^{\perp}$ actually globalizes to the sheaf $R^{1}\ppi^{\sm}_{*}\ql$ that leads us to consider the family $\ff^{\xi}: \overline{\kjj}^{\,\xi}\to \kbb$. Computations in examples suggest that we should also consider the family $\ppi^{[n]}:\kxx^{[n]}\to \kbb$. 
The confluence of these ideas then leads to theorem \ref{quotient main}. 
Moreover, by theorem \ref{quotient main} and the decomposition theorem \ref{cj support}, we can deduce directly another reformulation of the purity hypothesis, without recourse to our earlier work \cite{chen decomposition}.  

\begin{customcor}{2}\label{purity reformulated}
Suppose that $\xx_{\gamma}^{M}$ is cohomologically pure for every proper Levi subgroup $M \in \cl(M_{0})$, then $\xx_{\gamma}$ is cohomologically pure if and only if the cohomology groups
$\ch^{q}\big(\jj_{!*}\BigWedge^{i}(R^{1}\ppi^{\sm}_{*}\ql)\big)_{0}$
are pure of weight $q+i$ for all $q$ and $i$, where $\jj:\kbb^{\sm}\to \kbb$ is the inclusion of the open subscheme over which $\ppi$ is smooth and $\ppi^{\sm}$ is the restriction of $\ppi$ to the inverse image of $\kbb^{\sm}$.
\end{customcor}

We then proceed to compare this new sheaf-theoretic reformulation of cohomological purity  with the one mentioned above. 
To this end, we compare the microlocal analysis of the primitive part, as defined in equation $(\ref{springer sheaf})$, with that of the complex $\jj_{!*}\BigWedge^{i}(R^{1}\ppi^{\sm}_{*}\ql)$. The results are summarized in theorem \ref{morse primitive} and \ref{xi morse group}. 
These theorems exhibit notable similarities: in both cases the strict $\delta$-strata parametrized by Levi subgroups in $\cl(M_{0})$ makes no contribution to the Morse groups. 
As a consequence of the microlocal analysis, we use techniques developed in our work \cite{chen root valuation} to deduce that:

\begin{customthm}{3}\label{local constance}

Both the primitive part of $H^{*}(\xx_{\gamma}, \ql)$ and the cohomologies of the $\xi$-stable quotient $\xx_{\gamma}^{0,\, \xi}/\rS$ depends only on the root valuation datum of $\gamma$. 

\end{customthm}

In our work \cite{chen purity}, we prove the purity hypothesis for totally ramified elements $\gamma\in \kg(F)$.
Corollary \ref{purity reformulated} and Theorem \ref{local constance} constitute the main ingredients for a potential generalization of this proof to arbitrary semisimple regular elements. What remains missing is a suitable generalization of Teissier’s partial compactification of the moduli space of unibranch plane curve singularities \cite{teissier monomial} to the setting of reduced plane curve singularities—a problem for which we currently have no clear approach.

\subsection*{Notations and conventions}

\paragraph{(1)} 

We consider only schemes which are separated and of finite type over $k$, unless stated otherwise. For a geometric point $x$ of $X$, we denote by $X_{\{{x}\}}$ the strict Henselization of $X$ at ${x}$.

\paragraph{(2)}

Let $X$ be a locally noetherian scheme, we denote by $X^{\reg}$ its regular locus and $X^{\rm sing}=X-X^{\reg}$ the singular locus. 
Let $f:X\to S$ be a morphism of schemes, we will denote by $\Delta_{f}$ the discriminant locus of $f$, and by $S^{\sm}$ the open subscheme $S-\Delta_{f}$ over which $f$ is smooth, and $f^{\sm}:X^{\rm sm}\to S^{\sm}$  the  restriction of $f$ to the inverse image of $S^{\sm}$. For any point $s\in S$, we denote by $X_{s}$ the fiber of $f$ at $s$.

\paragraph{(3)}

Let $X$ be a locally noetherian scheme over $k$, we denote by $\cd_{c}(X,\ql)$ the derived category of $\ql$-constructible sheaves on $X$ defined by Deligne \cite{weil2}, \S1.1.2 and \S1.1.3. 
We denote by $\cp(X,\ql)$ the strict full subcategory of $\cd_{c}^{b}(X,\ql)$ consisting of the perverse $\ql$-sheaves. It is the heart of a $t$-structure on $\cd_{c}^{b}(X,\ql)$ associated to the autodual perversity. It can be shown that $\cp(X,\ql)$ is a notherian and artinian abelian category (cf. \cite{bbdg}). 
We denote by ${}^{p}\ch^{0}$ the cohomological functor $\tau_{\ge 0}\tau_{\le 0}$ from $\cd_{c}(X,\ql)$ to its heart $\cp(X,\ql)$. 
For $n\in \bz$, we denote by $[n]$ the shift to the left by $n$ in $\cd_{c}(X,\ql)$ and ${}^{p}\ch^{n}$ the composite ${}^{p}\ch^{0}\circ [n]$, we denote also by $(n)$ the Tate twist by $\mathbf{Q}_{\ell}(n)=\varprojlim_{m} \mu_{\ell^{m}}^{\otimes n}$.

\paragraph{(4)}

Let $A=k[\![x,y]\!]/(f(x,y))$ and $C^{\circ}=\spf(A)$ be a germ of plane curve singularity, its \emph{$\delta$-invariant} is defined as $\delta(C^{\circ}):=\dim_{k}(A^{\flat}/A)$, where $A^{\flat}$ is the normalization of $A$.

\paragraph{(5)}

Let $C$ be a projective algebraic curve over $k$, its \emph{$\delta$-invariant} $\delta(C)$ is defined as the sum of the $\delta$-invariants of its singularities, its \emph{arithmetic genus} is the invariant $p_{a}(C)=\dim H^{1}(C, \co_{C})$, and its \emph{geometric genus} $p_{g}(C)$ is the sum of the genus of the irreducible components of the normalization of $C$.

\paragraph{(6)}

Let $A$ be a split maximal torus of $G$, we denote the root system of $G$ with respect to $A$  by $\Phi(G,A)$. For any subgroup $H$ of $G$ which is stable under the conjugation of $A$, we note $\Phi(H,A)$ for the roots appearing in $\lie(H)$. 

\paragraph{(7)}

We use the $(G,M)$ notation of Arthur \cite{a}. Let $T$ be a maximal torus of $G$, let $\cf(G, T)$ be the set of parabolic subgroups of $G$ containing $T$, let $\cl(G,T)$ be the set of Levi subgroups of $G$ containing $T$. For every $M\in \cl(G,T)$, we denote by $\cp(G, M)$ the set of parabolic subgroups of $G$ whose Levi factor is $M$, by $\cl(G, M)$ the set of Levi subgroups of $G$ containing $M$, and by $\cf(G, M)$ the set of parabolic subgroups of $G$ containing $M$. If the context is clear, we will simplify the notations by omitting $G$.

Let $X^*(M)=\Hom(M, \bbg_m)$ and $X_{*}(M)=\Hom(X^{*}(M), \bz)$.
Let $\ka_M^{*}=X^*(M)\otimes\br$ and $\ka_M=X_{*}(M)\otimes\br$. The restriction $X^{*}(M)\to X^{*}(T)$ induces an injection $\ka_{M}^{*}\hookrightarrow \ka_{T}^{*}$. Let $(\ka_{T}^{M})^{*}$ be the subspace of $\ka_{T}^{*}$ generated by the roots in $\Phi(M,T)$. We have the decomposition in direct sums
$$
\ka_{T}^{*}=(\ka_{T}^{M})^{*}\oplus \ka_{M}^{*}.
$$

The canonical pairing 
$
X_{*}(T)\times X^{*}(T)\to \bz
$ 
can be extended bilinearly to $\ka_{T}\times \ka_{T}^{*}\to \br$. For $M\in \cl(G, T)$, we can embed $\ka_{M}$ in $\ka_{T}$ as the orthogonal subspace to $(\ka_{T}^{M})^{*}$. Let $\ka_{T}^{M}\subset \ka_{T}$ be the subspace orthogonal to $\ka_{M}^{*}$. We have the dual decomposition
$$
\ka_{T}=\ka_{M}\oplus \ka_{T}^{M}.
$$
More generally, for $L,M\in \cf(G,T),\,M\subset L$, we also have a decomposition
$$
\ka_{M}^{G}=\ka_{L}^{G}\oplus \ka_{M}^{L}.
$$
Let $\pi_{M,L},\,\pi_{M}^{L}$ be the projections to the two factors. If the context is clear, we also simplify it to $\pi_{L},\,\pi^{L}$.
For an element $\xi\in \ka_{M}^{G}$, we denote $\xi_{L}=\pi_{L}(\xi)$ and $\xi^{L}=\pi^{L}(\xi)$.

\paragraph{(8)}

Let $
\nu: \xx\to \bz
$
be the mapping $[g]\in \xx\mapsto \val(\det(g))$, it is a fact that $\xx^{n}:=\nu^{-1}(n), n\in \bz,$ are all the connected components of the affine grassmannian $\xx$, and that all of them are isomorphic to the central connected component $\xx^{0}$.  
Intersecting with the affine Springer fiber $\xx_{\gamma}$, we get all its connected components $\xx_{\gamma}^{n}:=\xx_{\gamma}\cap \xx^{n}$, $n\in \bz$. 

More generally, let $M$ be a Levi subgroup of $G$, we take a split maximal torus $A$ of $M$, and define $\Lambda_{M}$ as the quotient of $X_{*}(A)$ by the coroot lattice of $M$. It is independent of the choice of $A$, this is the algebraic fundamental group introduced by Borovoi \cite{bo}. According to Kottwitz \cite{k1}, we have a canonical homomorphism
\begin{equation}\tag{2}\label{indexM}
\nu_{M}: M(F)\to \Lambda_{M},
\end{equation}
which is characterized by the following properties: it is trivial on the image of $M_{\mathrm{sc}}(F)$ in $M(F)$ ($M_{\mathrm{sc}}$ is the simply connected cover of the derived group of $M$), and its restriction to $A(F)$ coincides with the composition $A(F)\to A(F)/A(\co)\cong X_{*}(A)\to \Lambda_{M}$. Moreover, the morphism (\ref{indexM}) is trivial on $M(\co)$, hence descends to a map
$$
\nu_{M}:\xx^{M}\to \Lambda_{M},
$$
whose fibers are the connected components of $\xx^{M}$. Similarly, intersecting with the affine Springer fiber $\xx^{M}_{\gamma}$, we get all of its connected components $\xx_{\gamma}^{M, \lambda}:=\xx^{M}_{\gamma}\cap \nu_{M}^{-1}(\lambda)$, $\lambda\in \Lambda_{M}$. 

{\small
\paragraph*{Acknowledgement} 

I would like to dedicate this work to the memory of Professor G\'erard Laumon, in fond remembrance of the warm and formative days of my mathematical youth.
I would also like to thank the Research Center for Mathematics and Interdisciplinary Sciences at Shandong University for its generous support during my visits.

}

\section{Families attached to the deformation of a non-irreducible spectral curve}

Let $\gamma\in\ggl_{d}(F)$ be a regular semisimple integral element, we construct a reduced but non-irreducible spectral curve $X_{\gamma}$, and consider the relative $\xi$-stable compactified Jacobian and the relative Hilbert scheme attached to a miniversal deformation of $X_{\gamma}$.
As we are interested in the geometry of the affine Springer fiber, we work with $k=\bc$.

\subsection{The miniversal deformation of a non-irreducible spectral curve}\label{deform XI}

Let $\gamma=\diag(\gamma_{1},\cdots, \gamma_{r})$ with $\gamma_{i}\in \ggl_{d_{i}}$ such that the characteristic polynomial of $\gamma_{i}$ is irreducible over $F$.
Let $A=\co[\gamma]$ and $A_{i}=\co[\gamma_{i}], i=1,\cdots, r$. Let $X_{\gamma}^{\circ}=\spec(A)$ and $X_{i}^{\circ}=\spec(A_{i})$, then $X_{\gamma}^{\circ}=\bigcup_{i=1}^{r}X_{i}^{\circ}$. 
Let $0$ be the unique singularity of $X_{\gamma}^{\circ}$. 
For each germ $X_{i}^{\circ}$, we complete it to a rational irreducible projective curve $X_{i}$ with unique singularity at $0$. Let $X_{\gamma}=\bigcup_{i=1}^{r}X_{i}$ be the gluing of $X_{i}$ at the point $0$, such that the germs $X_{i}^{\circ}$ glue together to get $X_{\gamma}^{\circ}$ and that $X_{\gamma}\backslash\{0\}$ is smooth. 

\begin{prop}

The arithmetic genus of $X_{\gamma}$ is  $p_{a}(X_{\gamma})=\delta_{\gamma}-r+1$, where $\delta_{\gamma}$ is the $\delta$-invariant of the singularity $X_{\gamma}^{\circ}$.

\end{prop}

\begin{proof}

Let $\phi: X_{\gamma}^{\flat}\to X_{\gamma}$ be the normalization, then $X_{\gamma}^{\flat}$ is the disjoint union of $r$ copies of $\bp^{1}$. Let $A^{\flat}$ be the normalization of $A$, we have an exact sequence
$$
0\to \co_{X_{\gamma}}\to \phi_{*}\co_{X_{\gamma}^{\flat}}\to A^{\flat}/A\to 0,
$$ 
from which we deduce 
$$
\chi(X_{\gamma})=\chi(X_{\gamma}^{\flat})-\delta_{\gamma}. 
$$
The assertion then follows. 
\end{proof}

As explained in \cite{melo fourier 1}, \S3.1, there exists an algebraization of a miniversal deformation of $X_{\gamma}$: 
$$
\ppi:(\mathfrak{X}, X_{\gamma})\to (\kbb, 0).
$$
This family contains several naturally occurring subfamilies:
For a subset $I\subset \{1,\cdots, r\}$,  let $X_{I}\subset X_{\gamma}$ be the union $\bigcup_{i\in I}X_{i}$. 
For a partition $\{1,\cdots, r\}=\bigsqcup_{j=1}^{l}I_{j}$ of length $l$, denoted $I_{\bullet}\vdash \{1,\cdots,r\}$, let $X_{I_{\bullet}}$ be the disjoint union $\bigsqcup_{j=1}^{l} X_{I_{j}}$.
For such a partition, we get a subfamily $\ppi_{I_{\bullet}}: \kxx_{I_{\bullet}}\to \kbb_{I_{\bullet}}$, in which each $X_{I_{j}}$ deforms independently of the others, while the intersection number between deformations of $X_{I_{j}}$ and $X_{I_{j'}}$ remains constant for all $j\neq j'$.
These subfamilies also appear in our work \cite{chen decomposition}, \S4.4, albeit in different form. 
In fact, we employ different globalizations of the same local situation: the restrictions of both  families--$\pi:\cc\to \cb$ and $\ppi:\kxx\to \kbb$--to a sufficiently small open neighborhood of $0$ are miniversal deformations of the singularity $\spf(\co[\gamma])$. Consequently, the local study in a neighborhood of $0$ carried out over there applies to our current setting.

For a subset $I\subset \{1,\cdots, r\}$, let $\gamma_{I}=\diag(\gamma_{i})_{i\in I}$ and let $A_{I}=\co[\gamma_{I}]$, then $X_{I}^{\circ}:=\spec(A_{I})=\bigcup_{i\in I}X_{i}^{\circ}$. 
For a partition $\{1,\cdots, r\}=\bigsqcup_{j=1}^{l}I_{j}$, let $A_{I_{\bullet}}=\prod_{j=1}^{l}A_{I_{l}}$ and let $X_{I_{\bullet}}^{\circ}=\spec(A_{I_{\bullet}})=\bigsqcup_{j=1}^{l}X_{I_{j}}^{\circ}$. 
Consider the deformation functor
$$
\defm_{X^{\circ}_{I_{\bullet}}\to X_{\gamma}^{\circ}}^{\rm top}:=\defm_{A\hookrightarrow A_{I_{\bullet}}}^{\rm top}: {\rm Art}_{k}\to {\rm Sets}, 
$$
which sends an artinian $k$-algebra $R$ to the set of isomorphism classes of homomorphisms of $R$-algebras $A_{R}\to A_{I_{\bullet}, R}$ whose reduction modulo $\km_{R}$ equals $A\hookrightarrow A_{I_{\bullet}}$, here $A_{R}$ is a flat deformation of $A$ over $R$ and $A_{I_{\bullet}, R}$ is a flat deformation of $A_{I_{\bullet}}$. 
Similarly, we can define the deformation functors $\defm_{X_{\gamma}^{\circ}}^{\rm top}:=\defm_{A}^{\rm top}$ and $\defm_{X_{I_{\bullet}}^{\circ}}^{\rm top}:=\defm_{A_{I_{\bullet}}}^{\rm top}$.
All these functors are pro-representable, and we have the forgetful morphisms
$$
\defm_{A\hookrightarrow A_{I_{\bullet}}}^{\rm top}\to \defm_{A}^{\rm top} 
\quad \text{and}\quad
\defm_{A\hookrightarrow A_{I_{\bullet}}}^{\rm top}\to \defm_{A_{I_{\bullet}}}^{\rm top}.
$$

\begin{prop}\label{formal smooth}
The functor $\defm_{A\hookrightarrow A_{I_{\bullet}}}^{\rm top}$ is formally smooth, and the morphism $$\defm_{A\hookrightarrow A_{I_{\bullet}}}^{\rm top}\to \defm_{A_{I_{\bullet}}}^{\rm top}
$$
is formally smooth.

\end{prop}

\begin{proof}

This is just the corollary 4.31 and 4.32 of \cite{chen decomposition}. 

\end{proof}

Let $\widehat{\kbb}_{I_{\bullet}}$ be the formal neibourhood of $0$ in $\kbb_{I_{\bullet}}$. By proposition 4.28 of \cite{chen decomposition}, we obtain that $\widehat{\kbb}_{I_{\bullet}}$ is the schematic image of the forgetful functor
$$
\defm_{A\hookrightarrow A_{I_{\bullet}}}^{\rm top}\to \defm_{A}^{\rm top}.
$$
The resulting morphism 
$$
\defm_{A\hookrightarrow A_{I_{\bullet}}}^{\rm top}\to\widehat{\kbb}_{I_{\bullet}}
$$ 
is finite and generically an isomorphism. 
The functor $\defm_{A\hookrightarrow A_{I_{\bullet}}}^{\rm top}$ is formally smooth by proposition \ref{formal smooth}, hence it is pro-representable by the normalization ${\widehat{\kbb}}_{I_{\bullet}}^{\flat}$ of $\widehat{\kbb}_{I_{\bullet}}$. 
With this, we can reformulate proposition \ref{formal smooth} as:

\begin{prop}\label{strata local}

The normalization ${\widehat{\kbb}}_{I_{\bullet}}^{\flat}$ of $\widehat{\kbb}_{I_{\bullet}}$ is formally smooth over $k$, and the morphism ${\widehat{\kbb}}_{I_{\bullet}}^{\flat} \to \prod_{j=1}^{l}\defm^{\rm top}_{X^{\circ}_{I_{j}}}$ is formally smooth.

\end{prop}

\begin{rem}\phantomsection\label{partition levi}

\begin{enumerate}[label=$(\arabic*)$, itemsep=2pt]
\item
For the normalization ${\widehat{\kbb}}_{I_{\bullet}}^{\flat} \to \widehat{\kbb}_{I_{\bullet}}$, the inverse image of $0\in \widehat{\kbb}_{I_{\bullet}}$ consists of exactly one point, it corresponds to the pair $A\hookrightarrow A_{I_{\bullet}}$, we denote it by $0_{I_{\bullet}}$.

\item

There is a bijection between the set of partitions of $\{1,\cdots, r\}$ and the set $\cl(M_{0})$ of Levi subgroups, given by
$$
\{1,\cdots, r\}=\bigsqcup_{j=1}^{l} I_{j}\mapsto M_{I_{\bullet}}:=\prod_{j=1}^{l}\gl\Big(\textstyle{\bigoplus_{i\in I_{j}}} F(\gamma_{i})\Big).
$$
It is helpful to keep it in mind when interpreting results that will show up.

\end{enumerate}

\end{rem}

\subsection{The relative $\xi$-stable compactified Jacobian of the deformation}

Let $X$ be a connected reduced projective algebraic curve over $k$ with planar singularities. 
Recall that the compactified Picard scheme $\overline{\pic}_{X/k}$ of $X$ parametrizes the isomorphism class of rank $1$ torsion-free simple coherent $\co_{X}$-modules \cite{ak}. Here a coherent $\co_{X}$-module $\mathcal{I}$ is said to be 
\begin{enumerate}[topsep=-2pt, itemsep=-2pt]
\item[---]
\emph{of rank} $1$, if $\ci$ has generic rank $1$ on every irreducible component of $X$,

\item[---]
\emph{torsion-free}, if $\mathrm{Supp}(\ci)=X$, and $\dim \mathrm{Supp}(\ci')=1$ for any subsheaf $\ci'$ of $\ci$,

\item[---]
\emph{simple}, if $\End_{\co_{X}}(\ci)=\bc$.

\end{enumerate}
\smallskip

\noindent Recall also that the \emph{degree} of a rank $1$ torsion-free coherent $\co_{X}$-module $\ci$ is defined as
$$
\deg(\ci)=\chi(\ci)-\chi(\co_{X}), 
$$
where $\chi(\ci)=h^{0}(X,\ci)-h^{1}(X,\ci)$ is the Euler-Poincar\'e characteristic of $\ci$. Since $\chi(\ci)$--and hence also $\deg(\ci)$--is constant under deformations, we can decompose $\overline{\pic}_{X/k}$ as
$$
\overline{\pic}_{X/k}=\bigsqcup_{n\in \bz} \overline{\pic}_{X/k}^{\, n},
$$
where $\overline{\pic}_{X/k}^{\,n}$ is the open and closed subscheme of $\overline{\pic}_{X/k}$ parametrizing the isomorphism class of rank $1$ torsion-free simple coherent $\co_{X}$-modules  of degree $n$.
Each connected component $\overline{\pic}_{X/k}^{\, n}$ satisfies the valuative criteria for properness; however, if $X$ is not irreducible, it may fail to be of finite type or separated.
In their work on the weighted fundamental lemma \cite{laumon lemme 1}, Chaudouard and Laumon introduced a $\xi$-stability condition on the Hitchin space. For the group $\gl_{d}$, by the Beauville-Narasimhan-Ramanan correspondence \cite{bnr}, this translates to a $\xi$-stability condition on the compactified Jacobian\footnote{While this notion was inspired by the work of Est\`eves \cite{esteves}, it is important to note that they are actually different.}.

\begin{defn}

Let $I$ be the set of irreducible components of $X$. A \emph{stability condition}  on $X$ is a tuple of real numbers $\xi=\{\xi_{X_{i}}\}_{i\in I}$, one for each irreducible component $X_{i}$ of $X$, such that $\sum_{i\in I}\xi_{X_{i}}=0$. The $\xi$-stability condition is said to be \emph{generic} if $\sum_{i\in J}\xi_{X_{i}}\notin \bz$ for any subset $J\subsetneq I$. 

\end{defn}

For a subset $J\subset I$, let $X_{J}$ be the subcurve of $X$ consisting of the irreducible components indexed by $J$. For a rank 1 torsion-free coherent sheaf $\ci$ on $X$, we denote by $\ci_{J}$ the maximal subsheaf of $\ci$ which is supported on $X_{J}$. It is clear that $\ci_{J}$ is torsion-free of rank $1$.

\begin{defn}

Let $\xi$ be a stability condition on $X$, let $\ci$ be a rank $1$ torsion-free coherent $\co_{X}$-module of degree $0$, it is said to be $\xi$-\emph{semistable} if  
$$
\deg(\ci_{J})+\sum_{j\in J } \xi_{j}\le 0, \quad \forall J\subsetneq I,
$$
it is said to be \emph{stable} if all the above inequalities are strict.

\end{defn}
  
If $\xi$ is generic, the notion of $\xi$-semistable coincides with that of $\xi$-stable as $\deg(\ci_{J})$ takes values in $\bz$. 
We denote by $\overline{J}{}_{X}^{\,\xi}$ the open subscheme of $\overline{\pic}^{\,0}_{X/k}$ parametrizing rank $1$ torsion-free simple $\co_{X}$-modules of degree $0$ which are $\xi$-stable. 
The construction of compactified Picard scheme extends naturally to a family, so is the notion of $\xi$-stability. 
Let $\ff: \overline{\pic}^{\,0}_{\kxx/\kbb}\to \kbb$ be the relative compactied Jacobian of the family $\ppi:\kxx\to \kbb$. Let $\xi$ be a generic stability condition on $X_{\gamma}$, by lemma-definition 5.3 of \cite{melo fine}, $\xi$ extends canonically to a generic stability condition on each geometric fiber of the miniversal deformation $\ppi:(\kxx,X_{\gamma})\to (\kbb,0)$, thereby inducing a notion of $\xi$-stability on $\overline{\pic}^{\,0}_{\kxx/\kbb}$.
Let $\overline{\kjj}^{\,\xi}$ be the $\xi$-stable locus of  $\overline{\pic}^{\,0}_{\kxx/\kbb}$, it is an open subscheme of the latter. Let $$\ff^{\xi}: \overline{\kjj}^{\,\xi}\to \kbb$$ be the restriction of $\ff: \overline{\pic}^{\,0}_{\kxx/\kbb}\to \kbb$ to $\overline{\kjj}^{\,\xi}$, we call it the $\xi$-stable relative compactified Jacobian of the family $\ppi$. 

\begin{thm}[Chaudouard-Laumon \cite{laumon lemme 1}, th\'eor\`eme 6.2.2]
\label{xi total}

Let $\xi$ be a generic stability condition on $X_{\gamma}$, then $\overline{\kjj}^{\,\xi}$ is smooth over $k$, and the morphism $\ff^{\xi}: \overline{\kjj}^{\,\xi}\to \kbb$ is proper and flat of dimension $p_{a}(X_{\gamma})$.

\end{thm}

By Deligne's Weil-II \cite{weil2}, the complex $R\ff^{\xi}_{*}\ql$ is pure of weight $0$. 
By the decomposition theorem of Beilinson, Bernstein, Deligne and Gabber \cite{bbdg}, corollaire 5.4.6, we have
$$
R\ff_{*}^{\xi}\ql=\bigoplus_{i} {}^{p}\ch^{i}(R\ff_{*}^{\xi}\ql)[-i].
$$
It is a difficult but important problem to determine the (shifted) perverse sheaves ${}^{p}\ch^{i}(R\ff_{*}^{\xi}\ql)$.
Let $\Delta\subset \kbb$ be the discriminant of the family $\ppi:\kxx\to \kbb$, let $\kbb^{\sm}=\kbb-\Delta$, then the restriction of $\ppi$--and hence also of $\ff^{\xi}$--to the inverse image of $\kbb^{\sm}$ is smooth, we denote them by $\ppi^{\sm}$ and $\ff^{\xi, \sm}$ respectively.  Let $\jj:\kbb^{\sm}\to \kbb$ be the inclusion.

\begin{thm}[Chaudouard-Laumon \cite{laumon lemme 2}, th\'eor\`eme 10.5.1] \label{cj support}

Let $\xi$ be a generic stability condition on $X_{\gamma}$, then 
$$
R\ff_{*}^{\xi}\ql=\bigoplus_{i=0}^{2p_{a}(X_{\gamma})} \jj_{!*}R^{i}\ff^{\xi,\sm}_{*}\ql[-i].
$$
In particular, we have
\begin{equation}\label{cj fiber}
H^{i}(\overline{J}{}^{\,\xi}_{X_{\gamma}}, \ql)=\bigoplus_{i'=0}^{i} \ch^{i-i'}(\jj_{!*}R^{i'}\ff^{\xi,\sm}_{*}\ql)_{0}, \quad i=0, \cdots, 2p_{a}(X_{\gamma}).
\end{equation}

\end{thm}

At the vicinity of $0$, the sheaf $R^{i}\ff^{\xi,\sm}_{*}\ql$ can be described explicitly.  
Let $\bar{\xi}_{0}$ be a geometric generic point of $\kbb_{\{0\}}$, we have the Milnor fiber $\kxx_{\{0\}}\times_{\kbb_{\{0\}}}\bar{\xi}_{0}$ for the singularity $X_{\gamma}^{\circ}$. Let $E=H^{1}(\kxx_{\{0\}}\times_{\kbb_{\{0\}}}\bar{\xi}_{0}, \ql)$, it is equipped with an alternating bilinear form $\pair{\, ,}: E\otimes E\to \ql$ coming from the intersection form on the homology of the Milnor fiber. Let $E^{\perp}$ be the radical of the pairing $\pair{\, ,}$, which then induces a non-degenerate alternating bilinear form on $E/E^{\perp}$, we denote it by $\pair{\, ,}_{\rm red}$.
The local monodromy group $\pi_{1}(\kbb_{\{0\}}-\Delta, \bar{\xi}_{0})$ acts on $E$ and preserves $E^{\perp}$.
They correspond to local systems $\ce$ and $\ce^{\perp}$ respectively on a sufficiently small open neighbourhood $U$ of $0$, the local system $\ce$ is also called the sheaf of vanishing cycles.  
It is a fact that $E/E^{\perp}$ is an irreducible $\pi_{1}(\kbb_{\{0\}}-\Delta, \bar{\xi})$-module \cite{waj}, so the local system $\ce/\ce^{\perp}$ is irreducible.

\begin{prop}\label{quotient globalized}

We have $R^{1}\ff^{\xi,\sm}_{*}\ql=R^{1}\ppi^{\sm}_{*}\ql$ and $R^{1}\ppi^{\sm}_{*}\ql|_{U-\Delta}=\ce/\ce^{\perp}$, hence $R^{i}\ff^{\xi,\sm}_{*}\ql$ is a globalization of $\BigWedge^{i}(\ce/\ce^{\perp})$ for all $i$.

\end{prop}
  
\begin{proof}

For $i=1,\cdots, r$, let $\partial_{i}$ be the boundary of the branch $X_{i}^{\circ}$, it is homeomorphic to the circle $\bs^{1}$. As the Milnor fibration preserves the boundary of the fibers, they correspond bijectively to the boundaries of the Milnor fiber $\kxx_{\{0\}}\times_{\kbb_{\{0\}}}\bar{\xi}_{0}$, and we retain the same notations for them. Then $E^{\perp}$ is generated by the associated cohomology classes $ [\partial_{1}], \cdots, [\partial_{r}]$. 
Topologically, our globalization of $X_{\gamma}^{\circ}$ to the spectral curve $X_{\gamma}$ is obtained by attaching a closed unit disk to each of the boundary $\partial_{i}$, hence the fiber $\kxx_{\bar{\xi}_{0}}$ is obtained topologically from the Milnor fiber by the same operations, so 
$$
H^{1}(\kxx_{\bar{\xi}_{0}},\ql)\cong H^{1}\big(\kxx_{\{0\}}\times_{\kbb_{\{0\}}}\bar{\xi}_{0},\ql\big)/\big\langle [\partial_{1}], \cdots, [\partial_{r}]\big\rangle=E/E^{\perp}. 
$$
This implies $R^{1}\ppi^{\sm}_{*}\ql|_{U-\Delta}=\ce/\ce^{\perp}$,  and the remaining assertions follow from this identification.
\end{proof}

\subsection{The relative punctual Hilbert scheme of the deformation}

For $n\in \bn$, let $\ppi^{[n]}:\kxx^{[n]}\to \kbb$ be the relative punctual Hilbert scheme of $\ppi$, where $\kxx^{[n]}$ represents the functor that assigns to each $\kbb$-scheme $S$ the set of closed subschemes of $\kxx\times_{\kbb}S$ that are flat and finite of degree $n$ over $S$. 
It is known that $\ppi^{[n]}$  is a proper morphism.

\begin{prop}[Migliorini-Shende-Viviani \cite{miglio 2}, theorem 4.16]

$\kxx^{[n]}$ is smooth over $k$.

\end{prop}

By Deligne's Weil-II \cite{weil2}, the complex $R\ppi^{[n]}_{*}\ql$ is pure of weight $0$. 
By the decomposition theorem of Beilinson, Bernstein, Deligne and Gabber \cite{bbdg}, corollaire 5.4.6, we have
$$
R\ppi^{[n]}_{*}\ql=\bigoplus_{i} {}^{p}\ch^{i}(R\ppi^{[n]}_{*}\ql)[-i].
$$
But unlike ${}^{p}\ch^{i}(R\ff^{\xi}_{*}\ql)$, the (shifted) perverse sheave ${}^{p}\ch^{i}(R\ppi^{[n]}_{*}\ql)$ may have extra supports.
Recall that for each partition $I_{\bullet}\vdash\{1,\cdots, r\}$ we have the subfamily $\ppi_{I_{\bullet}}: \kxx_{I_{\bullet}}\to \kbb_{I_{\bullet}}$ of the family $\ppi$. 
By construction, in the subfamily $\ppi_{I_{\bullet}}$ each of the subcurve $X_{I_{j}}$ can be deformed independently of the others, and the number of intersections between $X_{I_{j}}$ and $X_{I_{j'}}$ remains constant for all $j\neq j'$. Hence there exists a partial resolution of $\ppi_{I_{\bullet}}$:
$$
\begin{tikzcd}
\kxx_{I_{\bullet}}^{\flat}\arrow[rd, "\pvarpi_{I_{\bullet}}"']\arrow[rr]&& \kxx_{I_{\bullet}}\arrow[ld, "\ppi_{I_{\bullet}}"]
\\
&\kbb_{I_{\bullet}}&
\end{tikzcd}
$$ 
which resolves the intersections between $X_{I_{j}}$ and $X_{I_{j'}}$ for all $j\neq j'$.
Let $\delta_{I_{\bullet}}$ be the total number of nodes being resolved. 
Let $\Delta_{I_{\bullet}}$ be the discriminant for $\pvarpi_{I_{\bullet}}$, let $\kbb_{I_{\bullet}}^{\sm}=\kbb_{I_{\bullet}}-\Delta_{I_{\bullet}}$, let $\jj_{I_{\bullet}}: \kbb_{I_{\bullet}}^{\sm}\to \kbb_{I_{\bullet}}$ be the inclusion.   
Let $\pvarpi_{I_{\bullet}}^{[n]}: \kxx_{I_{\bullet}}^{\flat, [n]}\to \kbb_{I_{\bullet}}$ be the relative punctual Hilbert scheme of $\ppi_{I_{\bullet}}$, and let $\pvarpi_{I_{\bullet}}^{[n], \sm}$ be its restriction to the inverse image of $\kbb_{I_{\bullet}}^{\sm}$.

\begin{thm}[Migliorini-Shende-Viviani \cite{miglio 2}, theorem 5.10] \label{hilbert support}

We have the isomorphism
$$
R\ppi^{[n]}_{*}\ql\cong \bigoplus_{I_{\bullet}\vdash \{1,\cdots, r\}} 
\bigoplus_{i}
\Big\{(\jj_{I_{\bullet}})_{!*}R^{i}\pvarpi_{I_{\bullet}, *}^{[n-\delta_{I_{\bullet}}], \sm}\ql [-2\delta_{I_{\bullet}}](-\delta_{I_{\bullet}})\Big\}[-i].
$$
In particular, for $i=0, \cdots, 2n$, we have
\begin{equation}\label{hilbert fiber}
H^{i}(X_{\gamma}^{[n]}, \ql)
\cong 
\bigoplus_{I_{\bullet}\vdash \{1,\cdots, r\}} 
\bigoplus_{i'=0}^{i-2\delta_{I_{\bullet}}}
\ch^{i-2\delta_{I_{\bullet}}-i'}\Big\{(\jj_{I_{\bullet}})_{!*}R^{i'}\pvarpi_{I_{\bullet},*}^{[n-\delta_{I_{\bullet}}], \sm}\ql \Big\}_{0}\otimes \ql(-\delta_{I_{\bullet}}).
\end{equation}

\end{thm}

\begin{rem}

It is instructive to observe the structural similarities between theorem \ref{hilbert support} and theorem 5.2 of our work \cite{chen decomposition}, which decomposes the complex $Rf_{n, *}\ql$ as direct sum of intersection complexes, where $f_{n}:\overline{\cp}_{n}\to \cb_{n}$ is the miniversal $\Lambda^{0}/n$-equivariant deformation of the finite abelian covering $\overline{P}_{n}$ of $\overline{P}_{C_{\gamma}}$ with Galois group $\Lambda^{0}/n$. 
In both theorems, the decomposition receives contributions from perverse sheaves supported on the stratum $\kbb_{I_{\bullet}}$ and $\cs_{I_{\bullet}}$ respectively, as $I_{\bullet}$ ranges over all the partitions of $\{1,\cdots,r\}$.
It is even more instructive to observe that $\kbb_{I_{\bullet}}$ and $\cs_{I_{\bullet}}$ are analytically isomorphic in a sufficiently small open neighborhood of $0$.

\end{rem}

The summand at the right hand side of the equation (\ref{hilbert fiber}) indexed by the partition $I_{\bullet}=\{1,\cdots, r\}$ is called the \emph{main term}, we will show that the remaining terms can be simplified.
Since they involve the fibers at $0$ of certain intersection complex, it suffices to study the behaviour of the intersection complex in a sufficiently small open neighbourhood $U$ of $0$.
By construction, the restriction of $\ppi:\kxx\to \kbb$ to $U$ is a miniversal deformation of the singularity $X_{\gamma}^{\circ}$.  
Let $\varphi_{I_{\bullet}}: \kbb_{I_{\bullet}}^{\flat}\to \kbb_{I_{\bullet}}$ be the normalization. 
Let $U_{I_{\bullet}}=U\cap \kbb_{I_{\bullet}}$ and $U_{I_{\bullet}}^{\flat}=\varphi_{I_{\bullet}}^{-1}(U_{I_{\bullet}})$. 
For a subset $I\subset \{1,\cdots,r\}$, let $\ppi_{I}:(\kxx_{I}, X_{I})\to (\kbb_{I}, 0)$ be an algebraization of a miniversal deformation of $X_{I}$, and let $\ff_{I}^{\xi}: \overline{\mathfrak{J}}{}_{I}^{\,\xi}\to 
\kbb_{I}$ be $\xi$-stable relative compactified Jacobian of the family $\ppi_{I}$, and let $\ppi^{[n]}_{I}:\kxx^{[n]}_{I}\to \kbb_{I}$ be the relative punctual Hilbert scheme.
By proposition \ref{strata local}, we have a formally smooth morphism ${\widehat{\kbb}}_{I_{\bullet}}^{\flat} \to \prod_{j=1}^{l}\defm^{\rm top}_{X^{\circ}_{I_{j}}}$, hence the pull-back of the family $\pvarpi_{I_{\bullet}}: \kxx_{I_{\bullet}}^{\flat}\to \kbb_{I_{\bullet}}$ along the normalization $U_{I_{\bullet}}^{\flat}\to U_{I_{\bullet}}$ is a versal deformation of the curve $X_{I_{\bullet}}=\bigsqcup_{j=1}^{l} X_{I_{j}}$, and we get  a smooth morphism $\phi_{I_{\bullet}}: U_{I_{\bullet}}^{\flat}\to \prod_{j=1}^{l}\kbb_{I_{j}}$. 
For an unordered partition $n=n_{1}+\cdots+n_{l}$, we get a product family 
$$
\ppi^{[n_{1},\cdots, n_{l}]}_{I_{\bullet}}:\textstyle{\prod_{j=1}^{l}}\kxx^{[n_{j}]}_{I_{j}}\to \textstyle{\prod_{j=1}^{l}}\kbb_{I_{j}}.
$$
The pull-back of the family $\pvarpi_{I_{\bullet}}^{[n]}: \kxx_{I_{\bullet}}^{\flat, [n]}\to \kbb_{I_{\bullet}}$ along the normalization $U_{I_{\bullet}}^{\flat}\to U_{I_{\bullet}}\subset \kbb_{I_{\bullet}}$ then coincides with the pull-back of the family $\bigsqcup_{n=n_{1}+\cdots+n_{l}} \ppi_{I_{\bullet}}^{[n_{1}, \cdots, n_{l}]}$ along the morphism $\phi_{I_{\bullet}}$.
Consequently, we get a factorization
\begin{equation}\label{factorize hilbert}
\varphi_{I_{\bullet}}^{*}(R\pvarpi_{I_{\bullet}, *}^{[n]}\ql)|_{U_{I_{\bullet}}^{\flat}}=
\bigoplus_{n=n_{1}+\cdots+n_{l}} 
\phi_{I_{\bullet}}^{*}\big(R\ppi_{I_{1},*}^{[n_{1}]}\ql\boxtimes \cdots \boxtimes R\ppi_{I_{l},*}^{[n_{l}]}\ql\big).
\end{equation}
For $j=1,\cdots, l$, let $\jmath_{I_{j}}: \kbb_{I_{j}}^{\sm}\to \kbb_{I_{j}}$ be the inclusion of the open subscheme over which $\ppi_{I_{j}}$ is smooth. 
Let $U_{I_{\bullet}}^{\flat, \sm}:=U_{I_{\bullet}}^{\flat}\times_{\kbb_{I_{\bullet}}}\kbb_{I_{\bullet}}^{\sm}$ and let $j_{I_{\bullet}}^{\flat}: U_{I_{\bullet}}^{\flat, \sm}\to U_{I_{\bullet}}^{\flat}$ be the inclusion.
Applying the functor $(j_{I_{\bullet}}^{\flat})_{!*}(j_{I_{\bullet}}^{\flat})^{*}$ to both sides of (\ref{factorize hilbert}), we get the equality of the main terms 
\begin{equation}\label{hilbert factorization main}
\big((j^{\flat}_{I_{\bullet}})_{!*}\varphi_{I_{\bullet}}^{*}R\pvarpi_{I_{\bullet}, *}^{[n], \sm}\ql\big)|_{U_{I_{\bullet}}^{\flat}}=
\bigoplus_{n=n_{1}+\cdots+n_{l}} 
\phi_{I_{\bullet}}^{*}\big((\jmath_{I_{1}})_{!*}R\ppi_{I_{1},*}^{[n_{1}], \sm}\ql\boxtimes \cdots \boxtimes (\jmath_{I_{l}})_{!*}R\ppi_{I_{l},*}^{[n_{l}],\sm}\ql\big). 
\end{equation}
Recall that the normalization $\varphi_{I_{\bullet}}: \kbb_{I_{\bullet}}^{\flat}\to \kbb_{I_{\bullet}}$ is finite and generically an isomorphism, and the inverse image of $0\in \kbb_{I_{\bullet}}$ consists of unique point $0_{I_{\bullet}}$. Applying the functor $(\varphi_{I_{\bullet}})_{*}$ to the equation (\ref{hilbert factorization main}), and take the fiber at $0$, we get
\begin{equation}\label{levi simplified}
\big((\jj_{I_{\bullet}})_{!*}R\pvarpi_{I_{\bullet}, *}^{[n], \sm}\ql\big)_{0}=\bigoplus_{n=n_{1}+\cdots+n_{l}} 
\big((\jmath_{I_{1}})_{!*}R\ppi_{I_{1},*}^{[n_{1}], \sm}\ql\big)_{0}\otimes \cdots \otimes \big((\jmath_{I_{l}})_{!*}R\ppi_{I_{l},*}^{[n_{l}],\sm}\ql\big)_{0}, 
\end{equation}
here for the right hand side we have used the fact that $\phi_{I_{\bullet}}$ is smooth. 
Combining the equations (\ref{hilbert fiber}) and (\ref{levi simplified}), we get
\begin{align}
H^{i}(X_{\gamma}^{[n]}, \ql)
\cong &
\bigoplus_{I_{\bullet}\vdash \{1,\cdots, r\}} 
\bigoplus_{i'=0}^{i-2\delta_{I_{\bullet}}}
\bigoplus_{\substack{n-\delta_{I_{\bullet}}=n_{1}+\cdots+n_{l}\\l=length(I_{\bullet})}} 
\bigoplus_{i'=i'_{1}+\cdots+i'_{l}} \nonumber
\\
&
\ch^{i-2\delta_{I_{\bullet}}-i'}\Big\{\big((\jmath_{I_{1}})_{!*}R^{i_{1}'}\ppi_{I_{1},*}^{[n_{1}], \sm}\ql \big)_{0}\otimes \cdots \otimes \big((\jmath_{I_{l}})_{!*}R^{i_{l}'}\ppi_{I_{l},*}^{[n_{l}], \sm}\ql)\big)_{0} \Big\}\otimes \ql(-\delta_{I_{\bullet}}). \label{hilbert factor}
\end{align}
Each of the factor $R\ppi_{I_{j},*}^{[n_{j}], \sm}\ql$ can be further simplified. Indeed, for a flat family of smooth projective irreducible curves $p:C\to S$, let $p^{[n]}: C^{[n]}\to S$ be the relative Hilbert scheme, we have Macdonald's formula \cite{macdonald}: 
\begin{equation*}\label{macdo}
Rp^{[n]}_{*}\ql=\bigoplus_{\substack{i+j\le n\\ i,j\ge 0}} \BigWedge^{i}R^{1}p_{*}\ql [-i-2j](-j).
\end{equation*}
Plug it into the equation (\ref{hilbert factor}), we get

\begin{thm}\label{hilbert}
We have the isomorphism
\begin{align*}
H^{i}(X_{\gamma}^{[n]}, \ql)
\cong &
\bigoplus_{I_{\bullet}\vdash \{1,\cdots, r\}} 
\bigoplus_{i'=0}^{i-2\delta_{I_{\bullet}}}
\bigoplus_{\substack{n-\delta_{I_{\bullet}}=n_{1}+\cdots+n_{l}\\l=length(I_{\bullet})}} 
\bigoplus_{\substack{i'=i'_{1}+\cdots+i'_{l}\\ i_{j}'\le n_{j}, \,j=1,\cdots, l}}
\bigoplus_{i_{j}'=i_{j}''+k_{j}, \, j=1, \cdots, l}
\\
&
\ch^{i-2(\delta_{I_{\bullet}}+k_{1}+\cdots+k_{l})-i'}\Big\{\big((\jmath_{I_{1}})_{!*}(\BigWedge^{i_{1}''}R^{1}\ppi_{I_{1},*}^{\sm}\ql)\big)_{0}\otimes \cdots \otimes \big((\jmath_{I_{l}})_{!*}(\BigWedge^{i_{l}''}R^{1}\ppi_{I_{l},*}^{\sm}\ql)\big)_{0} \Big\}
\\
&\otimes \ql(-\delta_{I_{\bullet}}-k_{1}-\cdots-k_{l}).
\end{align*}

\end{thm}

\begin{rem}

We have explained a bijection between the set of partitions of $\{1,\cdots,r\}$ and the set $\cl(M_{0})$ of Levi subgroups containing $M_{0}$ in remark \ref{partition levi}. 
With this in mind and with the isomorphism (\ref{cj fiber}), Theorem \ref{hilbert} can be regarded as a global analogue of the Harder-Narasimhan reduction for the affine Springer fibers as explained in \cite{chen xi} and \cite{chen truncate}.

\end{rem}

\section{Consequence for the affine Springer fibers}

We explore the implications of theorem \ref{cj support} and \ref{hilbert} for the affine Springer fiber $\xx_{\gamma}$.
For this, we describe the $\xi$-stable compactified Jacobian $\overline{J}_{X_{\gamma}}^{\,\xi}$ and the punctual Hilbert scheme $X_{\gamma}^{[n]}$ in terms of $\xx_{\gamma}$. 
Let $\rT\subset G_{F}$ be the centralizer of $\gamma$ and let $\rS$ be the maximal $F$-split subtorus of $\rT$, it is clear that $\rS\cong \bbg_{m}^{r}$.

\subsection{The $\xi$-stable compactified Jacobian and the $\xi$-stabe quotient of $\xx_{\gamma}$}

The $\xi$-stable compactified Jacobian $\overline{J}_{X_{\gamma}}^{\,\xi}$ of $X_{\gamma}$ can be described in terms of $\xx^{0, \,\xi}_{\gamma}/\rS$, following ideas of Laumon \cite{laumon springer}. 
To begin with, notice that any rank $1$ torsion-free coherent $\co_{X_\gamma}$-module $\ci$ satisfies $\ci|_{X_{\gamma}\backslash\{0\}}\cong \co_{X_{\gamma}\backslash\{0\}}$ as  $X_{\gamma}\backslash\{0\}$ is the disjoint union of $r$ copies of $\ba^{1}$, and it can be obtained by glueing a rank $1$ torsion-free coherent sheaf $I$ on $X_{\gamma}^{\circ}$ with the sheaf $\co_{X_{\gamma}\backslash\{0\}}$.
Let $\overline{\mathbf{P}}{}^{\natural}_{X_{\gamma}}$ be the functor which associates to an affine $k$-scheme $S$ the groupo\"id of couples $(\ci,\iota)$, where $\ci$ is a rank $1$ torsion-free coherent $\co_{X_\gamma\times S}$-module of degree $0$, and $\iota:\ci|_{(X_\gamma\backslash \{0\})\times S}\cong \co_{(X_\gamma\backslash \{0\})\times S}$ is a trivialization of $\ci$ over $(X_\gamma\backslash \{0\})\times S$. 
Then $\overline{\mathbf{P}}{}^{\natural}_{X_{\gamma}}$ is representable by a $k$-scheme, and if we forget the rigidification $\iota$ we get the quotient stack 
$
\big[\,\overline{\mathbf{P}}{}^{\natural}_{X_{\gamma}}/\bbg_{m}^{r}\big]$. Note that it strictly contains $ \overline{\pic}^{\,0}_{X_{\gamma}/k}$ because we haven't required $\ci$ to be simple.  
On the other hand, the affine Springer fiber $\xx_{\gamma}$ parametrizes the lattices $L$ in $F^{n}$ such that $\gamma L\subset L$. Such lattices can be naturally identified with sub-$\co[\gamma]$-modules of finite type in $E:=F[\gamma]\cong F^{d}$, or equivalently rank $1$ torsion-free coherent sheaves on $X_{\gamma}^{\circ}$. 
We can glue them with the sheaf $\co_{X_{\gamma}\backslash\{0\}}$ along $(X_{\gamma}\backslash\{0\})\times_{X_{\gamma}} X_{\gamma}^{\circ}=\spec(E)$, to get rank $1$ torsion-free  coherent $\co_{X_{\gamma}}$-modules.  
With this construction, we get a morphism
$$
\xx^{0}_{\gamma}\to \overline{\bp}{}^{\natural}_{X_{\gamma}}.
$$
The same argument as in the proof of Proposition 2.3.1 in \cite{laumon springer} applies and yields the following:

\begin{prop}\label{cj springer}

The morphism $
\xx^{0}_{\gamma}\to \overline{\bp}{}^{\natural}_{X_{\gamma}}
$ is a universal homeomorphism.

\end{prop}

We proceed to compare the notion of $\xi$-stability on $\xx^{0}_{\gamma}$ and $\overline{\bp}{}^{\natural}_{X_{\gamma}}$.
Let me recall the $\xi$-stability condition on $\xx_{\gamma}$ as defined in \cite{chen xi} and \cite{chen truncate}. 
Let $M_{0}$ be the centralizer of $\rS$ in $G$, it is the minimal Levi subgroup of $G$ containing $\rT$. 
For $M\in \cl(M_{0})$, the inclusion $M\hookrightarrow G$ induces a closed immersion of $\xx^{M}$ in $\xx^{G}$.  
For $P=MN\in \cf(M_{0})$, we have the retraction
$$
f_{P}: \xx^{G}\to \xx^{M}, \quad [g]=[nmk]\mapsto [m],
$$
where $g=nmk$ with $n\in N(F), m\in M(F), k\in K$ is the Iwasawa decomposition. 
Moreover, we have a unique map\footnote{Our definition differs from the conventional one by a minus sign. } $H_{M}: M(F)\to \ka_{M}$  satisfying
$$
\chi(H_{M}(m))=\val(\chi(m)),\quad \forall\, \chi\in X^{*}(M),\,m\in M(F),
$$
which is actually a group homomorphism.
It is invariant under the right $K$-action, so it induces a map from $\xx^{M}$ to $\ka_{M}$, still denoted by $H_{M}$. We define $H_{P}:\xx\to \ka_{M}$ as the composition 
$$
H_{P}:\xx\xrightarrow{f_{P}}\xx^{M}\xrightarrow{H_{M}}\ka_{M}.
$$
It can be shown that $H_{P}$ coincides with the composition
$$
H_{P}:\xx\xrightarrow{f_{P}}\xx^{M}\xrightarrow{\nu_{M}}\Lambda_{M}\to \ka_{M}.
$$

The map $H_{P}$ has the following remarkable property.
There is a notion of adjacency among the parabolic subgroups in $\cp(M)$: Two parabolic subgroups $P_{1}=MN_{1},\,P_{2}=MN_{2}\in \cp(M)$ are said to be \emph{adjacent} if both of them are contained in a parabolic subgroup $Q=LN_{Q}$ such that $L\supset M$ and $\rk(L)=\rk(M)+1$. Given such an adjacent pair, we define an element $\beta_{P_{1},P_{2}}\in \Lambda_{M}$ in the following way: Consider the collection of elements in $\Lambda_{M}$ obtained from coroots of $A$ in $\kn_{1}\cap \kn_{2}^{-}$, we define $\beta_{P_{1},P_{2}}$ to be the minimal element in this collection, i.e. all the other elements are positive integral multiples of it. Note that $\beta_{P_{2},P_{1}}=-\beta_{P_{1},P_{2}}$, and if $M=A$, then $\beta_{P_{1},P_{2}}$ is the unique coroot which is positive for $P_{1}$ and negative for $P_{2}$. We denote also by $\beta_{P_{1},P_{2}}$ for its image in $\ka_{M}$ if no confusion is caused.

\begin{prop}[Arthur \cite{a}, Lemma 3.6]\label{arthur orthogonal}

Let $P_{1},\,P_{2}\in \cp(M)$ be two adjacent parabolic subgroups. For any $x\in \xx$, we have
$$
H_{P_{1}}(x)-H_{P_{2}}(x)=n(x,P_{1},P_{2})\cdot \beta_{P_{1},P_{2}},
$$
for some $n(x, P_{1}, P_{2})\in \bz_{\ge 0}$.

\end{prop}

For any point $x\in \xx$, we write $\ec_{M}(x)$ for the convex hull in $\ka_{M}$ of the $H_{P}(x),\,P\in \cp(M)$. For any $Q\in \cf(M)$, we denote by $\ec_{M}^{Q}(x)$ the face of $\ec_{M}(x)$ whose vertices are $H_{P}(x),\,P\in \cp(M),\,P\subset Q$.

\begin{defn}

Let $\xi\in \ka_{M}^{G}$ be a generic element.
A point $x\in \xx$ is said to be $\xi$-\emph{stable} if the polytope $\pi^{G}\big(\ec_{M}(x)\big)$ contains $\xi$. 

\end{defn}

By lemma 4.1 of \cite{chen truncate}, the $\xi$-stable locus
$$
\xx^{\xi}=\big\{x\in \xx\mid \xi\in \pi^{G}\big(\ec_{M}(x)\big)\big\}
$$
is an open sub-ind-$k$-scheme of $\xx$. Let 
$$
\xx_{\gamma}^{\xi}=\xx_{\gamma}\cap \xx^{\xi},
$$
then it is an open subscheme of $\xx_{\gamma}$.
The morphism $\rT(F)\xrightarrow{H_{M}} X_{*}(M)$ is surjective, hence the connected components of $\xx_{\gamma}^{\xi}$ can be translated to each other by elements of $\rT(F)$. 
In particular, all the connected components of $\xx_{\gamma}^{\xi}$ are isomorphic to each other. 
Moreover, for different choices of generic element $\xi,\,\xi'\in \ka_{M}^{G}$, the corresponding $\xx_{\gamma}^{\xi},\,\xx_{\gamma}^{\xi'}$ can be translated to each other by elements of $\rT(F)$. Hence $\xx_{\gamma}^{\xi}$ doesn't depend on the choice of $\xi$.  

The $\xi$-stability condition can be reformulated in linear algebraic terms. 
We can identify $\ka_{M_{0}}=\ka_{\rS}$. 
Moreover, we have the isomorphism
$$
\rS\cong H^{0}\big(\spec(F(\gamma)), \bbg_{m}\big)=\prod_{i=1}^{r} H^{0}\big(\spec(F(\gamma_{i})), \bbg_{m}\big). 
$$
An element $\xi\in \ka_{M_{0}}$ can be written as $\xi=(\xi_{i})_{i=1}^{r}$ with $\xi_{i}\in H^{0}\big(\spec(F(\gamma_{i})), \br\big)=\br$, hence it can be identified with a stability condition on $X_{\gamma}$.  
For a lattice $L\subset F^{d}$, we define its \emph{index} as 
$$
\ind(L)=[\co^{d}:L\cap \co^{d}]-[L:L\cap \co^{d}].
$$

\begin{lem}\label{xi plain}

Let $\xi\in \ka_{M_{0}}$ be generic with $\xi_{1}+\cdots+\xi_{r}=0$. 
A point $x_{L}\in \xx_{\gamma}^{0}$ corresponding to a lattice $L\subset F^{d}$ is $\xi$-stable if and only if 
\begin{equation*}
\textstyle{\sum_{i\in I}\xi_{i}}<\ind\big(L\cap \textstyle{\bigoplus_{i\in I}}F(\gamma_{i})\big), \quad \forall\, I\subsetneq \{1,\cdots, r\}.
\end{equation*}

\end{lem}

\begin{proof}

This is a generalization of proposition 3.4 of \cite{chen xi}.
We identify $\Lambda_{M_{0}}=\bz^{r}$ such that the map $\nu_{M_{0}}:M_{0}(F)\to\Lambda_{M_{0}}=\bz^{r}$  is just
$$(m_{1},\cdots, m_{r})\in M_{0}(F)\mapsto \big(\val(\det(m_{1})), \cdots, \val(\det(m_{r}))\big)\in \bz^{r}.$$ 
For $\tau\in \kss_{r}$, we associate the parabolic subgroup $P_{\tau}$ which is the stabilizer of the flag
$$
F(\gamma_{\tau(1)})\subsetneq F(\gamma_{\tau(1)})\oplus F(\gamma_{\tau(2)})\subsetneq\cdots \subsetneq F(\gamma_{\tau(1)})\oplus\cdots\oplus F(\gamma_{\tau(r)})=F^{d}.
$$ 
In this way, we get a bijection between $\cp(M_{0})$ and  $\kss_{r}$. Let $H_{P_{\tau}}(x_{L})=(n_{1},\cdots, n_{r})\in \bz^{r}$, then
$$
n_{\tau(1)}+\cdots+n_{\tau(i)}=\ind(L\cap (F(\gamma_{\tau(1)})\oplus \cdots \oplus F(\gamma_{\tau(i)}))), \quad i=1,\cdots,r. 
$$
So $x_{L}$ is $\xi$-stable if and only if 
$$
\xi_{\tau(1)}+\cdots+\xi_{\tau(i)}<\ind(L\cap (F(\gamma_{\tau(1)})\oplus \cdots \oplus F(\gamma_{\tau(i)}))), \quad \text{for all } \tau\in \kss_{r} \text{ and } i=1,\cdots,r, 
$$
which is equivalent to the condition in the lemma.

\end{proof}


\begin{prop}\label{cj springer stable}

Let $\xi\in \ka_{M_{0}}$ be generic with $\xi_{1}+\cdots+\xi_{r}=0$. The homeomorphism $\xx^{0}_{\gamma}\to \overline{\bp}{}^{\natural}_{X_{\gamma}}$ induces a universal homeomorphism
$$
\xx_{\gamma}^{0, \,\xi}/\rS\to \overline{J}^{\,\xi}_{X_{\gamma}}.
$$
Consequently, $\xx_{\gamma}^{0, \,\xi}/\rS$ is proper over $k$.

\end{prop}

\begin{proof}

To begin with, we show that the notion of $\xi$-stability on $\xx_{\gamma}$ and $\overline{\bp}{}^{\natural}_{X_{\gamma}}$ coincides. 
Let $L\subset F^{d}$ be a lattice of index $0$ satisfying $\gamma \cdot L\subset L$, and let $\cl$ be the  coherent $\co_{X_{\gamma}}$-module obtained by the glueing of $L$ and $\co_{X_{\gamma}\backslash\{0\}}$ along $X_{\gamma}^{\circ}$. Then
\begin{equation}\label{degree index}
\deg(\cl)=\chi(\cl)-\chi(\co_{X_{\gamma}})=-\ind(L). 
\end{equation}
By lemma \ref{xi plain}, the point $x_{L}$ on $\xx_{\gamma}$ corresponding to the lattice $L$ is $\xi$-stable if and only if 
\begin{equation}\label{stability plain}
\ind\big(L\cap \textstyle{\bigoplus_{i\in I}}F(\gamma_{i})\big)>\sum_{i\in I}\xi_{I}, \quad \forall\, I\subsetneq \{1,\cdots,r\}.
\end{equation}
Let $L_{I}=L\cap \textstyle{\bigoplus_{i\in I}}F(\gamma_{i})$ and let $\cl_{I}$ be the coherent $\co_{X_{\gamma_{I}}}$-module obtained by the glueing of $L_{I}$ with $\co_{X_{\gamma_{I}}\backslash\{0\}}$ along $X_{\gamma_{I}}^{\circ}$.
Then $\cl_{I}$ is the maximal subsheaf of $\cl$ which is supported on $X_{I}$.
The point on $\overline{\bp}{}^{\natural}_{X_{\gamma}}$ corresponding to $\cl$ with the obvious rigidification is $\xi$-stable if and only if 
$$
\deg(\cl_{I})+\sum_{i\in I}\xi_{i}<0, \quad \forall\, I\subsetneq \{1,\cdots,r\}.
$$
By the equation (\ref{degree index}), this is equivalent to
$$
-\ind(L_{I})+\sum_{i\in I}\xi_{i}<0,  
$$ 
which is the same as (\ref{stability plain}), whence the coincidence of $\xi$-stability condition.

Let $\cl$ be a $\xi$-stable rank 1 torsion-free coherent sheaf of degree $0$ on $X_{\gamma}$, then it must be simple. Otherwise, by proposition 1 of \cite{esteves}, it can be decomposed as $\cl=\cl_{I}\oplus \cl_{I^{c}}$ for some $I\subsetneq \{1,\cdots,r\}$.
Let $L$ (respectively $L_{I}$ and $L_{I^{c}}$) be the lattice corresponding to $\cl$ (respectively $\cl_{I}$ and $\cl_{I^{c}}$). 
 By lemma \ref{xi plain}, we have
$$
\ind(L_{I})>\textstyle{\sum_{i\in I}}\xi_{i}
\quad 
\text{and}
\quad
\ind(L_{I^{c}})>\textstyle{\sum_{i\in I^{c}}}\xi_{i},
$$ 
whence a contradiction
$$
0=\ind(L)=\ind(L_{I})+ \ind(L_{I^{c}})>\textstyle{\sum_{i=1}^{r}}\xi_{i}=0.
$$
As a consequence, the $\xi$-stable locus $\overline{\bp}^{\natural,\,\xi}_{X_{\gamma}}$ of $\overline{\bp}^{\natural}_{X_{\gamma}}$ parametrizes the $\xi$-stable rank $1$ torsion-free simple $\co_{X_{\gamma}}$-modules $\cl$ of degree $0$ equipped with a rigidification $\iota:\cl|_{X_{\gamma}\backslash\{0\}}\cong \co_{X_{\gamma}\backslash\{0\}}$. Forgetting the rigidification, we get
$$
\overline{\bp}^{\natural,\,\xi}_{X_{\gamma}}/\bbg_{m}^{r}=\overline{J}_{X_{\gamma}}^{\,\xi}.
$$
Combined with proposition \ref{cj springer} and the above comparison of $\xi$-stability, we get the universal homeomorphism between $\xx^{0,\, \xi}_{\gamma}/\rS$ and $\overline{J}_{X_{\gamma}}^{\,\xi}$. We deduce that $\xx^{0,\, \xi}_{\gamma}/\rS$ is proper because $\overline{J}_{X_{\gamma}}^{\,\xi}$ is proper by theorem \ref{xi total}.

\end{proof}

\subsection{The punctual Hilbert scheme and its local analogues}

The punctual Hilbert scheme $X^{[n]}_{\gamma}$ has a very nice structure. 
Let $X_{\gamma}^{\circ, [n]}$ be the local Hilbert scheme parametrizing the length $n$ quotient algebras $Q$ of $\co[\gamma]$. Take the kernel of the morphism $\co[\gamma]\twoheadrightarrow Q$, we get a rank $1$ torsion-free coherent $\co[\gamma]$-module $I\subset \co[\gamma]$ of index $n$, this sets up an isomorphism between $X_{\gamma}^{\circ, [n]}$ and a specific truncated affine Springer fiber
\begin{equation}\label{local hilbert springer}
X_{\gamma}^{\circ, [n]}\cong \xx_{\gamma}(n):=\big\{L\subset F^{d} \text{ a lattice}\mid \gamma L\subset L \text{ and } L\subset \co[\gamma] \text{ of index }n\big\}.
\end{equation}
We label the irreducible components of $X_{\gamma}\backslash\{0\}$, all of which are isomorphic to $\ba^{1}$, as $\ba_{1}, \cdots, \ba_{r}$.  Then we have the decomposition of $X_{\gamma}^{[n]}$ according to the support of the quotient sheaves
\begin{align*}
X_{\gamma}^{[n]}&=\bigsqcup_{\substack{n=n_{0}+n_{1}+\cdots+n_{r}\\ n_{0},\cdots, n_{r}\ge 0}} X_{\gamma}^{\circ, [n_{0}]}\times \ba_{1}^{[n_{1}]}\times\cdots\times \ba_{r}^{[n_{r}]}
\\
&=\bigsqcup_{\substack{n=n_{0}+n_{1}+\cdots+n_{r}\\ n_{0},\cdots, n_{r}\ge 0}} X_{\gamma}^{\circ, [n_{0}]}\times \ba^{n_{1}}\times\cdots\times \ba^{n_{r}}.
\end{align*}
As each $\ba_{i}$ is open in its closure, for $0\le a\le n$, the union
$$
X_{\gamma, a}^{[n]}:=\bigsqcup_{a\le n_{0}\le n} \bigsqcup_{\substack{n=n_{0}+n_{1}+\cdots+n_{r}\\ n_{1},\cdots, n_{r}\ge 0}} X_{\gamma}^{\circ, [n_{0}]}\times \ba_{1}^{[n_{1}]}\times\cdots\times \ba_{r}^{[n_{r}]}
$$
is a closed subscheme of $X_{\gamma}^{[n]}$, and each subscheme $X_{\gamma}^{\circ, [a]}\times \ba_{1}^{[n_{1}]}\times\cdots\times \ba_{r}^{[n_{r}]}$ with $n_{1}+\cdots+n_{r}=n-a$ is open in $X_{\gamma,a}^{[n]}$. Hence we have the long exact sequence
$$
\cdots \to \bigoplus_{n_{1}+\cdots+n_{r}=n-a}H_{c}^{*}(X_{\gamma}^{\circ, [a]}\times \ba_{1}^{[n_{1}]}\times\cdots\times \ba_{r}^{[n_{r}]}, \ql) \to H^{*}(X_{\gamma, a}^{[n]}, \ql)\to H^{*}(X_{\gamma, a+1}^{[n]}, \ql)\xrightarrow{\delta} \cdots,
$$
which can be simplified as
\begin{equation}\label{hilbert exact 0}
\cdots \to H^{*-2(n-a)}(X_{\gamma}^{\circ, [a]}, \ql)^{\oplus \rp_{r}(n-a)}\otimes \ql(-(n-a)) \to H^{*}(X_{\gamma, a}^{[n]}, \ql)\to H^{*}(X_{\gamma, a+1}^{[n]}, \ql)\xrightarrow{\delta} \cdots,
\end{equation}
where $\rp_{r}(n-a)$ is the number of unordered partitions $n-a=n_{1}+\cdots+n_{r}, n_{i}\ge 0, i=1,\cdots,r$.
By performing the computations inductively on $a$, we obtain an expression of $H^{*}(X_{\gamma}^{[n]}, \ql)$ in terms of the cohomological groups $H^{*}(X_{\gamma}^{\circ, [a]}, \ql), a=0, \cdots, n$.

\begin{prop}\label{hilbert local global} 

$X_{\gamma}^{[n]}$ are cohomologically pure for $1\le n\le m$ if and only if $X_{\gamma}^{\circ, [n']}$ are cohomologically pure for $1\le n'\le m$.

\end{prop}

\begin{proof}


For the sufficiency, we can show that $X_{\gamma}^{[n]}$ is cohomologically pure by using the long exact sequence (\ref{hilbert exact 0}) inductively, starting from $a=n$ and running downward to $a=0$.  

For the necessity, we proceed by induction on $n'$. The assertion is trivial for $n'=1$ as $X_{\gamma}^{\circ, [1]}=\mathrm{pt}$.
Suppose that $X_{\gamma}^{\circ, [n']}$ are cohomologically pure for $1\le n'\le m'-1$, we need to show that $X_{\gamma}^{\circ, [m']}$ is cohomologically pure.
This is equivalent to the cohomological purity of $X_{\gamma,m'}^{[m']}$, since it coincides with $X_{\gamma}^{\circ, m'}$.
For the last assertion, we proceed by induction to show that $X_{\gamma,a}^{[m']}$ are cohomologically pure for $0\le a\le m'$, this will finish the proof. 
To begin with, $X_{\gamma,0}^{[m']}=X_{\gamma}^{[m']}$ is cohomologically pure by assumption. 
Consider the long exact sequence (\ref{hilbert exact 0}) for $n=m'$ and $0\le a\le m'-1$. 
By assumption, $X_{\gamma}^{\circ, [a]}$ is cohomologically pure, hence the group
$$
H^{i-2(n-a)}(X_{\gamma}^{\circ, [a]}, \ql)^{\oplus \rp_{r}(n-a)}\otimes \ql(-(n-a))
$$
is pure of weight $i$. 
However, the group $H^{i-1}(X_{\gamma, a+1}^{[n]}, \ql)$ is mixed of weight $\le i-1$ by Deligne's Weil-II \cite{weil2}, as $X_{\gamma, a+1}^{[n]}$ is closed in $X_{\gamma}^{[n]}$ which is projective.
This forces the boundary map $\delta$ to be $0$. 
The cohomological purity of $X_{\gamma,a}^{[m']}$ then implies the cohomological purity of $X_{\gamma,a+1}^{[m']}$, and this finishes the induction step.

\end{proof}

\begin{prop}\label{springer n}

If $X_{\gamma}^{\circ, [n]}$ is cohomologically pure for all $n\in \bn$, then $\xx_{\gamma}$ is cohomologically pure. 

\end{prop}

\begin{proof}

By the isomorphism (\ref{local hilbert springer}), $\xx_{\gamma}(n)\cong X_{\gamma}^{\circ, [n]}$ is also cohomologically pure for all $n\in \bn$.
Let $N=\ind(\co[\gamma])$, for $n$ such that $d|(N+n)$, the translation on the left of $\xx_{\gamma}(n)$ by $$\lambda_{n}:=\diag(\ep^{-(N+n)/d}, \cdots, \ep^{-(N+n)/d})$$ 
is contained in $\xx_{\gamma}^{0}$, and we have
$$
\xx_{\gamma}^{0}=\varinjlim_{n\equiv -N\mod d} \lambda_{n}\cdot \xx_{\gamma}(n).
$$
The cohomological purity of all the $\xx_{\gamma}(n)$ implies the cohomological purity of $\xx_{\gamma}^{0}$, hence of $\xx_{\gamma}$.

\end{proof}

\subsection{Implications of the decomposition theorems}

We are now ready to prove theorem \ref{quotient main}.

\begin{prop}\label{quotient necessity}

If $\xx_{\gamma}$ is cohomologically pure, then $\xx^{0, \xi}_{\gamma}/\rS$ is also cohomologically pure. 

\end{prop}

\begin{proof}

Suppose that $\xx_{\gamma}$ is cohomologically pure. Then the quotient stack $[\xx_{\gamma}/\rS]$ is also cohomologically pure, because we have a fibration $[\xx_{\gamma}/\rS]\to [\mathrm{pt}/\rS]$ with fiber $\xx_{\gamma}$ and $B\rS=[\mathrm{pt}/\rS]$ is cohomologically pure.
As $\xx^{0, \xi}_{\gamma}/\rS$ is open in $[\xx_{\gamma}/\rS]$, we have a morphism
\begin{equation}\label{sub}
H^{*}_{c}(\xx^{0, \xi}_{\gamma}/\rS, \ql)\to H^{*}\big([\xx_{\gamma}/\rS], \ql\big),
\end{equation}
and by restriction we get another morphism
\begin{equation}\label{quotient}
H^{*}\big([\xx_{\gamma}/\rS], \ql\big)\to H^{*}(\xx^{0, \xi}_{\gamma}/\rS, \ql).
\end{equation}
By proposition \ref{cj springer stable}, the quotient $\xx_{\gamma}^{0, \xi}/\rS$ is proper over $k$, hence
$$
H^{*}_{c}(\xx^{0, \xi}_{\gamma}/\rS, \ql)=H^{*}(\xx^{0, \xi}_{\gamma}/\rS, \ql),
$$
and the composite of the morphism (\ref{sub}) and (\ref{quotient}) is identity on $H^{*}(\xx^{0, \xi}_{\gamma}/\rS, \ql)$. Hence $H^{*}(\xx^{0, \xi}_{\gamma}/\rS, \ql)$ appears as a direct summand of $H^{*}\big([\xx_{\gamma}/\rS], \ql\big)$, and the cohomological purity of the latter implies that of the former.

\end{proof}

\begin{rem}
Using the Harder-Narasimhan reduction of $\xx_{\gamma}$ as developed in \cite{chen xi} and \cite{chen truncate}, we can show that if $\xx^{M,\, \xi^{M}}_{\gamma}/\rS$ are cohomologically pure for all the Levi subgroups $M\in \cl(M_{0})$, then the stack $\big[\xx_{\gamma}/\rS\big]$ is cohomologically pure. However, it seems that this statement yields no information about the purity of $\xx_{\gamma}$ itself.

\end{rem}

\begin{proof}[Proof of theorem $\ref{quotient main}$]


By Proposition~\ref{springer n}, together with the isomorphism~(\ref{local hilbert springer}), and Proposition~\ref{quotient necessity}, we have established the following implications of cohomological purity (abbreviated as C. P.):
$$
\text{C. P. of }\xx_{\gamma}(n) \text{ for all n} \Rightarrow \text{C. P. of }\xx_{\gamma}\Rightarrow \text{C.P. of }\xx_{\gamma}^{0,\,\xi}/\rS.  
$$
It remains to prove the other direction.

We proceed by induction on the rank of $G$.  The theorem is trivial for the group $\gl_{1}$.
By assumption, $\xx_{\gamma}^{M}$ are cohomologically pure for all the proper Levi $M\in \cl(M_{0})$. By induction, this implies that all the $\xi$-stable quotients $(\xx_{\gamma}^{M})^{0, \,\xi^{M}}/\rS$ are cohomologically pure. 
Let $M=M_{I_{\bullet}}, I_{\bullet}$ a partition of $\{1,\cdots, r\}$ of length $l\ge 2$, then
$$
(\xx_{\gamma}^{M})^{0, \,\xi^{M}}/\rS=\prod_{j=1}^{l}\Big(\xx_{\gamma_{I_{j}}}^{M_{I_{j}}}\Big)^{0, \,\xi^{M_{I_{j}}}}\big/(\rS\cap M_{I_{j}}).
$$
Applying proposition \ref{cj springer stable} and  theorem \ref{cj support} to each of the above factor, since each one is cohomologically pure, we obtain that for all $q$ and $i$ all the fibers 
$$
\ch^{q}(\jmath_{I, !*}R^{i}\ff_{I,*}^{\xi,\sm}\ql)_{0}=\ch^{q}(\jmath_{I, !*}\BigWedge^{i}R^{1}\ppi^{\sm}_{I,*}\ql)_{0}, \quad I\subsetneq \{1,\cdots,r\},
$$
are pure of weight $q+i$.
If $\xx_{\gamma}^{0, \,\xi}/\rS$ is cohomologically pure, by proposition \ref{cj springer stable} and theorem \ref{cj support} again, all the fibers 
$$
\ch^{q}(\jj_{!*} R^{i}\ff^{\xi,\sm}_{*}\ql)_{0}=\ch^{q}(\jj_{!*}\BigWedge^{i}R^{1}\ppi^{\sm}_{*}\ql)_{0}
$$
are pure of weight $q+i$ for all $q$ and $i$.
Consequently, at the right hand side of the isomorphism in theorem \ref{hilbert}, all the items are pure of the right weight, so $X_{\gamma}^{[n]}$ are cohomologically pure for all $n$. 
By proposition \ref{hilbert local global}, together with the isomorphism (\ref{local hilbert springer}), $X_{\gamma}(n)$ are cohomologically pure for all $n$. 
By proposition \ref{springer n}, $\xx_{\gamma}$ is cohomologically pure. 
This finishes the proof in the other direction.

\end{proof}

\section{Microlocal analysis for the affine Springer fibers}

We move on to compare the new sheaf-theoretic reformulation of the purity of the affine Springer fibers with a previous one obtained in our work \cite{chen decomposition}. For this, we make microlocal analysis of the relevant intersection complexes which appear in the various decomposition theorems.

\subsection{Recall on the microlocal analysis for the compactified Jacobian}\label{cj microlocal}

Let $C_{\gamma}$ be an irreducible rational projective curve with unique singularity isomorphic to $\spec(\co[\gamma])$, let $\overline{P}_{\gamma}$ be its compactified Jacobian. According to Laumon \cite{laumon springer}, we have a universal homeomorphism
$$
\Lambda\backslash \xx_{\gamma}\doteq \overline{P}_{\gamma}.
$$
In particular, they have the same \'etale cohomology.
Let $\pi:(\cc, C_{\gamma})\to (\cb, 0)$ be an algebraization of a miniversal deformation of $C_{\gamma}$, let $f:\overline{\cp}\to \cb$ be the relative compactified Jacobian of $\pi$. Let $j:\cb^{\sm}\to \cb$ be the inclusion of the open subscheme over which $\pi$ and hence $f$ is smooth, let $\pi^{\sm}$ and $f^{\sm}$ be the restriction of $\pi$ and $f$ respectively over the inverse image of $\cb^{\sm}$. 

\begin{thm}[Support theorem of Ng\^o \cite{ngo}]\label{ngo support}

For the family $f:\overline{\cp}\to \cb$, we have
$$
Rf_{*}\ql=\bigoplus_{i=0}^{2\delta_{\gamma}} j_{!*}R^{i}f^{\sm}_{*}\ql[-i].
$$
In particular,
$$
H^{i}(\Lambda\backslash \xx_{\gamma}, \ql)=H^{i}(\overline{P}_{\gamma}, \ql)=\bigoplus_{i'=0}^{i}\ch^{i-i'}(j_{!*}R^{i'}f^{\sm}_{*}\ql)_{0}.
$$

\end{thm}

In this section, we recall our work \cite{chen root valuation} on the microlocal analysis of the fiber $\ch^{*}(j_{!*} R^{i}f^{\sm}_{*}\ql)_{0}$ with the stratified Morse theory of Goresky and MacPherson \cite{gm morse}.
For simplicity, we denote $\cf=R^{1}\pi^{\sm}_{*}\ql=R^{1}f^{\sm}_{*}\ql$ and $\cf^{i}=\BigWedge^{i}\cf$ for $i=0, \cdots, 2\delta_{\gamma}$. 
We assume that $\dim(\cb)=\tau_{\gamma}>\delta_{\gamma}$, which excludes only the case that $\gamma=\diag(\gamma_{1}, \gamma_{2})\in \ggl_{2}(F)$ such that $\val(\gamma_{1}-\gamma_{2})=1$, for which the geometry of $\xx_{\gamma}$ is well known. 
Let $\big(\{\cc_{\alpha}\}_{\alpha\in \overline{\Omega}}, \{\cb_{\beta}\}_{\beta\in \Omega}\big)$ be the canonical Whitney stratification for the family $\pi:\cc\to \cb$ (cf. \cite{mather stratification}, proposition 10.1). 
Then the collection of the strict $\delta$-strata is part of the Whitney stratification $\{\cb_{\beta}\}_{\beta\in \Omega}$.

Let $H_{0}$ be a generic linear subspace of $\cb=\ba^{\tau_{\gamma}}$ of dimension $\delta_{\gamma}+1$ such that it intersects transversally with the tangent cone at $0$ of all the $\delta$-strata. We then cut off the locus where its intersection with the union of all the $\delta$-strata is not transversal, to get a dense open subscheme $\cb'\subset H_{0}$ containing $0$. 
Consider the restriction of the family $f$ to the inverse image of $\cb'$:
$$
f':\overline{\cp}'=\overline{\pic}{}^{0}_{\cc'/\cb'}\to \cb'.
$$
We show that the subspace $\overline{\cp}{}'$ is nonsingular over $\bc$, and Ng\^o's support theorem holds for $f'$, i.e.
\begin{equation*}
Rf'_{*}\ql=\bigoplus_{i=0}^{2\delta_\gamma} j'_{!*}(R^{\,i}{f'}^{\sm}_{*}\ql)[-i],
\end{equation*}
where $j':(\cb')^{\sm}\to \cb'$ is the natural inclusion. 
Consequently, let $B_{0}(\ep)$ be a sufficiently small open ball around $0\in \cb=\ba^{\tau_{\gamma}}$, we have the decomposition
\begin{equation*}
H^{i}(\overline{P}_{C_{\gamma}}, \ql)=\bigoplus_{i'=0}^{i}\ch^{i-i'}\big(j'_{!*}\cf^{i'}\big)_{0}=\bigoplus_{i'=0}^{i}H^{i-i'}\big(\cb'\cap B_{0}(\ep), j'_{!*}\cf^{i'}\big), \quad i=0,\cdots, 2\delta_{\gamma}.
\end{equation*}

We can make further analysis of the cohomologies at the right hand side of the above equation. 
We take a flag of generic linear subspaces
$$
\{0\}= H_{\delta_{\gamma}+1}\subsetneq H_{\delta_{\gamma}}\subsetneq \cdots \subsetneq H_{2}\subsetneq H_{1}\subsetneq H_{0}, \quad \mathrm{codim}_{H_{0}}(H_{k})=k \text{ for all }1\le k \le \delta_{\gamma}+1,
$$
for which two properties concerning the local polar varieties of the relevant Whitney strata, proposition 4.11 and 4.14 of \cite{chen root valuation}, hold. 
For $k=1, \cdots, \delta_{\gamma}$, let $\pi_{k}:H_{0}\to \bc^{k}$ be the projection with kernel $H_{k}$, let $t_{k}\in \bc^{k}$ be a geometric generic point sufficiently close to $0$. The projection $\pi_{1}$ induces a generic projection
$
p_{1}:\cb'\cap B_{0}(\ep)\to \bc^{1}, 
$
and the projection $\pi_{k+1}$ induces an affine map 
$$
p_{k+1}:\cb'\cap B_{0}(\ep)\cap\pi_{k}^{-1}(t_{k})\to \bc^{1}, \quad \text{for }k=1,\cdots, \delta_{\gamma}.
$$ 
The points $\{t_{k}\}_{k=1}^{\delta_{\gamma}}$ have been taken such that $\cb'\cap B_{0}(\ep)\cap \pi_{k+1}^{-1}(t_{k+1})$ coincides with the fiber of $p_{k+1}$ at a generic point in ${\rm Im}(p_{k+1})$. 
For simplicity, we denote 
$$F^{k}=\begin{cases} \cb'\cap B_{0}(\ep),& \text{for }k=0,\\
\cb'\cap B_{0}(\ep)\cap\pi_{k}^{-1}(t_{k}), &\text{for }k=1,\cdots, \delta_{\gamma},\end{cases}
$$ 
it is of dimension $\delta_{\gamma}+1-k$. 
Note that $F^{\delta_{\gamma}}$ is a generic line in $\cb'\cap B_{0}(\epsilon)$, and a generic geometric point on it is the same as a generic geometric point of $\cb'\cap B_{0}(\epsilon)$, we will denote it by $\bar{\xi}$.  
We then replace the projection $p_{k+1}$ by a stratified Morse function 
$$
\varphi_{k+1}:F^{k}\to \br^{1}, \quad k=0, \cdots, \delta_{\gamma},
$$ 
with respect to the stratification of $F^{k}$ induced from the  canonical Whitney stratification $\{\cb_{\beta}\}_{\beta\in \Omega}$, which approximates sufficiently well to the function ${\rm Re}(p_{k+1})$. 
Without loss of generality, we can assume that the critical points of $\varphi_{k+1}$ have distinct critical values.
For $r\in \br$, let $F^{k}_{<r}=\varphi_{k+1}^{-1}(-\infty, r)$.
Let $v_{k}$ be the smallest critical value of $\varphi_{k+1}$, then
$$
H^{*}\big(F^{k}_{<v_{k}},  j'_{!*}\cf^{i} \big)= \begin{cases}
\cf^{i}_{\bar{\xi}}, &\text{ for }k=\delta_{\gamma},\\
H^{*}\big(F^{k+1}, j'_{!*}\cf^{i}\big), &\text{ for }k=\delta_{\gamma}-1,\cdots, 0,\end{cases}
$$
and the cohomology group $H^{*}\big(F^{k},  \jj'_{!*}\cf^{i} \big)$ can be built up from it and the Morse groups at the critical points $v$ of $\varphi_{k+1}$ via the long exact sequence
{\small
\begin{equation*}
\cdots \to H^{*-1}\big(F^{k}_{<v-a}, j_{!*}\cf^{i}\big) \xrightarrow{\partial} H^{*}\big(F^{k}_{<v+a}, F^{k}_{<v-a}; j_{!*}\cf^{i}\big)  \to  H^{*}\big(F^{k}_{<v+a},  j_{!*}\cf^{i}\big) \to H^{*}\big(F^{k}_{<v-a}, j_{!*}\cf^{i}\big) \to \cdots.
\end{equation*}
}

\noindent In this way, we can compute inductively the cohomologies $H^{*}\big(F^{k},  j'_{!*}\cf^{i} \big)$, starting from $\cf^{i}_{\bar{\xi}}$ and arriving at $H^{*}\big(F^{0},  j'_{!*}\cf^{i} \big)=\ch^{*}(j_{!*}\cf^{i})_{0}$.

\begin{thm}[\cite{chen root valuation}, prop. 3.8]\label{morse group}

Let $x\in F^{k}\cap\cb_{\beta}$ be a critical point of $\varphi_{k+1}$ with critical value $v$, let $\lambda$ be the Morse index of the restriction of $\varphi_{k+1}$ to $F^{k}\cap\cb_{\beta}$.
Let $a$ be a sufficiently small positive real number.

\begin{enumerate}[topsep=0pt, itemsep=0pt, label=$(\arabic*)$]

\item

If $\cb_{\beta}$ is not one of the strict $\delta$-strata, then
$$
H^{*}\big(F^{k}_{<v+a}, F^{k}_{<v-a}; j'_{!*}\cf^{i}\big)=0,\quad \text{for all }i.
$$

\item

If $\cb_{\beta}$ is one of the strict $\delta$-strata $\cb_{\delta}^{\circ}$ with $\delta\ge 1$, let $\widetilde{\cc}_{x}$ be the normalization of $\cc_{x}$ and let 
$$
H^{*}(\jac_{\widetilde{\cc}_{x}}, \ql)=\bigoplus_{n=0}^{2(\delta_{\gamma}-\delta)}\boldsymbol{\Lambda}_{x}^{n}[-n],
$$ 
then for all $i$ we have
$$
H^{q}\big(F^{k}_{<v+a}, F^{k}_{<v-a}; j'_{!*}\cf^{i}\big)=
\begin{cases}
\boldsymbol{\Lambda}_{x}^{{i-\delta}}, & \text{ if }q=\lambda+\delta,
\\
0,&\text{otherwise}.
\end{cases}
$$

\end{enumerate}

\end{thm}

\begin{proof}

We make a brief recall of the proof. 
Let $N_{x}$ be a transversal slice to $\cb_{\beta}$ at $x$ in $\cb$, let $\rp:(N_{x}, x)\to (\ba^{1}, 0)$ be a generic projection. 
Let $\bar{\eta}_{0}$ be a geometric generic point of $\ba^{1}_{\{0\}}$, let $R\Phi_{\bar{\eta}_{0}}$ be the vanishing cycle functor with respect to $\rp$.
By the stratified Morse theory of Goresky and MacPherson \cite{gm morse}, up to a shift by the Morse index, the Morse group in question is equal to the vanishing cycle
$$
(R^{q}\Phi_{\bar{\eta}_{0}}(j_{!*}R^{i}f_{*}^{\sm}\ql))_{x}.
$$

In case $(1)$, we can show that the composite $\rp\circ f|_{N_{x}}: f^{-1}(N_{x})\to \ba^{1}$ is smooth at the vicinity of $0\in \ba^{1}$. The local acyclicity of the smooth morphism then implies that
$$
R\Phi'_{\bar{\eta}_{0}}\bq_{\ell}\big|_{f^{-1}(x)}=0,
$$
where $R\Phi'_{\bar{\eta}_{0}}$ is the vanishing cycle functor for the morphism $\rp\circ f|_{N_{x}}$. 
By proper base change theorem, this implies
$
(R\Phi_{\bar{\eta}_{0}}Rf_{*}\ql)_{x}=0. 
$
With the support theorem of Ng\^o's, we get
$$
(R^{q}\Phi_{\bar{\eta}_{0}}j_{!*}R^{i}f^{\sm}_{*}\ql)_{x}=0, \quad \text{for all } q \text{ and }i.
$$
This finishes the sketch of proof in case $(1)$.

In case $(2)$, the fiber $\cc_{x}$ is an irreducible projective curve with $\delta$ nodes. The restriction of $\pi:\cc\to \cb$ to the inverse image of $N_{x}$ is then a miniversal deformation of all the nodes on $\cc_{x}$. 
As explained in \cite{chen root valuation}, the restriction of the family $f:\overline{\cp}\to \cb$ to the inverse image of $N_{x}$ can be described as follows:
Let $C_{xy}$ be the irreducible projective rational curve over $\bc$ with a unique ordinary double point as singularity, let $\overline{P}_{xy}$ be its compactified Jacobian, then $\overline{P}_{xy}\cong C_{xy}$.   
With the gluing construction of Laumon \cite{laumon springer}, we have a fibration $\overline{\jac}_{\cc_{x}}\to \jac_{\widetilde{\cc}_{x}}$ with fiber homeomorphic to the $\delta$-fold product $\overline{P}_{xy}\times \cdots \times\overline{P}_{xy}$. For simplicity, we denote it by
\begin{equation}\label{devi j}
\begin{tikzcd}\underbrace{\overline{P}_{xy}\times \cdots \times\overline{P}_{xy}}_{\delta} \arrow[r] &\overline{\jac}_{\cc_{x}}\arrow[d]\\
& \jac_{\widetilde{\cc}_{x}}.
\end{tikzcd}
\end{equation}
Let $\hat{f}_{xy}:\widehat{\cp}_{xy}\to \widehat{\ba}_{0}^{1}$ be the miniversal deformation of $\overline{P}_{xy}$. 
The composite $p_{k+1}\circ f:=p_{k+1}\circ (f|_{f^{-1}(N_{x})})$ is then equivalent to
\begin{equation}\label{local model 0}
\underbrace{\widehat{\cp}_{xy}\times \cdots \times\widehat{\cp}_{xy}}_{\delta}
\xrightarrow{\hat{f}_{xy}\times\cdots\times \hat{f}_{xy}} \underbrace{\widehat{\ba}_{0}^{1}\times \cdots \times\widehat{\ba}_{0}^{1}}_{\delta}
\xrightarrow{\sum_{i=1}^{\delta}a_{i}} \widehat{\ba}_{0}^{1}.
\end{equation}
As $\overline{P}_{xy}\cong C_{xy}$, the restriction of the deformation $\hat{f}_{xy}:\widehat{\cp}_{xy}\to \widehat{\ba}_{0}^{1}$ to an open neighbourhood of the singular point can be described explicitly as $\bc^{2}\to \bc: (x,y)\mapsto xy$,
and the above composite is nothing but
\begin{equation*}
\bc^{2\delta}\to \bc,\quad (x_{1}, y_{1}, \cdots,x_{\delta},y_{\delta})\mapsto \sum_{i=1}^{\delta}x_{i}y_{i},
\end{equation*}
which describes the miniversal deformation of an ordinary quadratic singularity. 
Let $R\overline{\Phi}_{\bar{\eta}_{0}}$ be the vanishing cycle functor for the composite (\ref{local model 0}), then $R\bar{\Phi}_{\bar{\eta}_{0}}\ql$ is supported at the point $\bar{c}:=(c,\cdots,c)$ with $c$ being the unique double point of $\overline{P}_{xy}$, and that
$$
(R\bar{\Phi}_{\bar{\eta}_{0}}\ql)_{\bar{c}}=\ql[-2\delta+1].
$$
With the fibration (\ref{devi j}) and the freeness property in \cite{ngo}, proposition 7.4.10, we obtain that 
$$
(R\Phi_{\bar{\eta}_{0}}Rf_{*}\ql)_{x}
=\boldsymbol{\Lambda}_{x}^{\bullet}[-2\delta+1].
$$
By the support theorem of Ng\^o, we conclude
$$
(R^{\delta-1}\Phi_{\bar{\eta}_{0}}(j_{!*}\cf^{i}))_{x}
=\boldsymbol{\Lambda}_{x}^{{i-\delta}},
$$ 
and all the other vanishing cycles vanish.

\end{proof}

\begin{rem}
For $i=1, \cdots, 2\delta_{\gamma}-1$, let $\bar{\imath}=\min\{i, 2\delta_{\gamma}-i\}$, the theorem implies that the critical points on the $\delta$-strata for $\delta>\bar{\imath}$ makes no contribution to the Morse groups of $j_{!*}\cf^{i}$, so
$$
\ch^{*}(j_{!*}\cf^{i})_{0}=H^{*}(F^{\delta_{\gamma}+1-\bar{\imath}}, j_{!*}\cf^{i}). 
$$
It is enough to iterate the above process for $k=\delta_{\gamma},\cdots, \delta_{\gamma}+1-\bar{\imath}$.

\end{rem}

\subsection{Microlocal analysis for the primitive part}\label{microlocal primitive}

According to Laumon \cite{laumon springer}, the affine Springer fiber $\xx^{0}_{\gamma}$ is homeomorphic to a $\Lambda^{0}$-torsor $\overline{P}_{C_{\gamma}}^{\natural}$ of the compactified Jacobian $\overline{P}_{C_{\gamma}}$. 
In our work \cite{chen decomposition}, for any $n\in \bn$, we consider $\Lambda^{0}/n$-coverings $\overline{P}_{n}$ of $\overline{P}_{C_{\gamma}}$, and their miniversal $\Lambda^{0}/n$-equivariant deformations $f_{n}:\overline{\cp}_{n}\to \cb_{n}$, where $\varpi_{n}:\cb_{n}\to \cb^{\circ}$ is an \'etale morphism onto an open subscheme $\cb^{\circ}\subset \cb$. 
They form a projective system and can be regarded as a substitute for the deformation of $\xx_{\gamma}$. 
By construction, we have
\begin{equation*}
H^{*}(\xx^{0}_{\gamma}, \ql)=H^{*}(\overline{P}_{C_{\gamma}}^{\natural}, \ql)=\varprojlim_{m}\varinjlim_{n}H^{*}(\overline{P}_{n}, \bz/\ell^{m}).
\end{equation*}
Let $j:\cb^{\sm}\to \cb$ be the inclusion of the open subscheme over which $\pi$ is smooth, let $\pi^{\sm}$ be the restriction of $\pi$ to the inverse image of $\cb^{\sm}$. 
Let $\cf=R^{1}\pi^{\sm}_{*}\ql$ and $\ce\subset \cf|_{U}$ the local system of vanishing cycles for the singularity of $C_{\gamma}$, where $U$ is a sufficiently small open neighbourhood of $0$, let $\ce^{\perp}$ be the orthogonal complement of $\ce$ with respect to the cup product of $\cf$. 
We obtain a decomposition of $Rf_{n, *}\ql$ as a direct sum of intersection complexes, and deduce from it that 
\begin{align*}
H^{i}\big({\xx}_{\gamma}^{0},\bq\big)=&
\bigoplus_{i'=0}^{i} 
{\rm Im}\Big\{\ch^{i-i'}\big(j_{!*}\BigWedge^{i'}\cf\big)_{0}\to  
\ch^{i-i'}\big(j_{!*}\BigWedge^{i'}(\cf/\ce^{\perp})\big)_{0} \Big\}\oplus \\
&
\bigoplus\text{ Similar terms for the Levi subgroups}. 
\end{align*}
The cohomological purity of $\xx_{\gamma}$ is then equivalent to that of its primitive part
\begin{equation*}
\bigoplus_{i'=0}^{i} 
{\rm Im}\Big\{\ch^{i-i'}\big(j_{!*}\BigWedge^{i'}\cf\big)_{0}\to  
\ch^{i-i'}\big(j_{!*}\BigWedge^{i'}(\cf/\ce^{\perp})\big)_{0} \Big\}.
\end{equation*}
We will make a microlocal analysis of it.
But the above expression is difficult to work with, as we know nothing about the intersection complex $j_{!*}\BigWedge^{i}(\cf/\ce^{\perp})$.
Instead, we go back to the family $f_{n}:\overline{\cp}_{n}\to \cb_{n}$. 
By proposition 5.10 of \cite{chen decomposition}, the primitive part is actually a limit:
\begin{equation}\label{primitive as limit}
{\rm Im}\Big\{\ch^{*}\big(j_{!*}\BigWedge^{i}\cf\big)_{0}\to  
\ch^{*}\big(j_{!*}\BigWedge^{i}(\cf/\ce^{\perp})\big)_{0} \Big\}=\Big(\varprojlim_{m}\varinjlim_{n} \ch^{*}(j_{n, !*}R^{i}f^{\sm}_{n, *}\bz/\ell^{m})_{0_{n}}\Big)\otimes_{\bz_{\ell}}\ql,
\end{equation}  
where $j_{n}:\cb_{n}^{\sm}\to \cb_{n}$ is the inclusion of the open subscheme over which $f_{n}$ is smooth.
By construction, $\varpi_{n}:\cb_{n}\to \cb^{\circ}$ is \'etale and $0_{n}\in \cb_{n}$ is the point lying over $0\in \cb_{n}$. 
Hence the iterated projections for the base $\cb$ can be carried over to $\cb_{n}$, and we can make microlocal analysis of $j_{n, !*}R^{i}f^{\sm}_{n, *}\bz/\ell^{m}$ with the same process as explained in \S\ref{cj microlocal}. 
Take limit, we get the Morse groups for the primitive part. 

 \subsubsection{Finite abelian coverings of the compactified Jacobian}\label{recall finite covering}
 
Let $C$ be an irreducible projective curve with planar singularities at $z_{j}, {j\in J}$. Following ideas of Laumon \cite{laumon springer}, we give a concrete description of the finite abelian coverings of the compactified Jacobian $\overline{P}_{C}$.  

To begin with, we recall Laumon's description of $\overline{P}_{C}$. 
Let $\phi:\widetilde{C}\to C$ be the normalization of $C$, let $\phi^{-1}(z_{j})=\{\tilde{z}_{i}\mid i\in I_{j}\}$, then the local branches of $C$ at $z_{j}$ are naturally parametrized by $I_{j}$. 
For each $j\in J$, let $A_{I_{j}}$ be the completed local ring of $C$ at $z_{j}$, let $E_{I_{j}}$ be the ring of fractions of $A_{I_{j}}$, let $\xx_{I_{j}}$ be the affine Springer fiber which parametrizes all the $A_{I_{j}}$-fractional ideals in $E_{I_{j}}$. It is equipped with the natural action of the $k$-group scheme $G_{I_{j}}$ such that $G_{I_{j}}(k)=E_{I_{j}}^{\times}/A_{I_{j}}^{\times}$. Let
\begin{equation}\label{springer notation}
\xx(C)=\prod_{j\in J} \xx_{I_{j}},\quad G(C)=\prod_{j\in J} G_{I_{j}},\quad I=\bigsqcup_{j\in J} I_{j}, \quad \Lambda(C)=\prod_{j\in J}\bz^{I_{j}}.
\end{equation}
The group $G(C)$ admits a d\'evissage
$$
1\to G^{0}(C)\to G(C)\to \Lambda(C)\to 0,
$$
which at the level of $k$-points is
$$
1\to \prod_{j\in J} \co_{E_{I_{j}}}^{\times}/A_{I_{j}}^{\times} \to 
\prod_{j\in J} E_{I_{j}}^{\times}/A_{I_{j}}^{\times}\to \prod_{j\in J}\bz^{I_{j}}\to 1.
$$
Then $G(C)$ acts on $\xx(C)$ with a dense open orbit such that we can consider $\xx(C)$ as a compactification of $G(C)$. 
On the other hand, $G^{0}(C)$ appears as the maximal connected affine subgroup of the Jacobian $\jac_{C}$ of $C$, and we have the d\'evissage of Chevalley
\begin{equation}\label{chevalley cj curve}
1\to G^{0}(C)\to \jac_{C}\xrightarrow{\phi^{*}} \jac_{\widetilde{C}}\to 1.
\end{equation}
Let $\overline{P}_{C}^{\natural}$ be the moduli space of simple torsion-free coherent sheaves $\cm$ over $C$ of generic rank $1$ and degree $0$, equipped with local rigidification $\cm|_{\spf(E_{I_{j}})}\cong E_{I_{j}}, j\in J$. 
Let $\xx^{0}(C)$ be the central connected component of $\xx(C)$, then we have a $k$-morphism
\begin{equation*}
\xx^{0}(C)\to \overline{P}_{C}^{\natural}
\end{equation*}
which sends $(M_{j}\subset E_{I_{j}})_{j\in J}$ to the torsion-free coherent $\co_{C}$-module $\cm$ of generic rank $1$ obtained by gluing $\co_{C\backslash \{c_{j}\}_{j\in J}}$ with the $M_{j}$'s. It is clear that we can twist the construction by an invertible sheaf on $C$, and  get a morphism
\begin{equation}\label{uniform cj general}
\big[\xx^{0}(C)\times \jac_{C}\big]/G^{0}(C)\to \overline{P}_{C}^{\natural}. 
\end{equation}
Forgetting local rigidifications, we get
\begin{equation}\label{uniform cj general descent}
\big[(\Lambda^{0}(C)\backslash \xx^{0}(C))\times \jac_{C}\big]/G^{0}(C)\to \overline{P}_{C}.  
\end{equation}

\begin{prop}[Laumon \cite{laumon springer}]\label{cj fibration general}

The morphisms $(\ref{uniform cj general descent})$ is a universal homeomorphisms. 
Consequently, we get a fibration
$\overline{P}_{C}\to \jac_{\widetilde{C}}$
with fiber homeomorphic to 
\begin{equation*}
\Lambda^{0}(C)\backslash \xx^{0}(C)\cong \prod_{j\in J}\Lambda_{I_{j}}^{0}\backslash\xx^{0}_{I_{j}},
\end{equation*} 
where $\Lambda_{I_{j}}^{0}=\ker\big\{\bz^{I_{j}}\xrightarrow{\rm sum} \bz\big\}$.

\end{prop}

By the general theory of finite abelian coverings as explained in theorem 4.2 of \cite{chen decomposition}, a finite abelian covering $\Phi: \overline{P}'\to \overline{P}_{C}$ is  determined by a finite subgroup scheme $\ct\subset \pic_{\overline{P}_{C}}=\pic_{C/k}$. 
For such a finite abelian covering, we restrict $\Phi$ to the inverse image of $\jac_{C}$ and get a finite abelian covering $\Phi':J'\to \jac_{C}$ with the same Galois group. 
As the singularities of $C$ are planar, according to Altman-Iarrobino-Kleiman \cite{aik} and Rego \cite{rego}, $\overline{P}_{C}$ is geometrically integral and $\jac_{C}$ is dense open in $\overline{P}_{C}$. 
Hence $J'$ is dense open in $\overline{P}'$. To get a concrete description of $\overline{P}'$, we need to know what is $J'$ and how it is compactified to get $\overline{P}'$.

The structure of $J'$ has been described in the proof of theorem 4.3 of \cite{chen decomposition}.
Let $\phi^{*}:\pic_{C/k}\to \pic_{\widetilde{C}/k}$ be the pull-back of invertible sheaves, then $\Phi':J'\to \jac_{C}$ is a trivial $\Ker(\phi^{*}|_{\ct})$-covering of a finite abelian covering $J_{C}'$ of $\jac_{C}$.  
Let $J'_{\widetilde{C}}$ be the finite abelian covering of $\jac_{\widetilde{C}}$ associated to the torsion subgroup $\phi^{*}(\ct)^{\red}$, here the superscript $^{\red}$ means the group scheme with reduced structure. 
With the d\'evissage (\ref{chevalley cj curve}), we obtain that $\pic_{\jac_{C}}=\pic_{\jac_{\widetilde{C}}}$, and that $J_{C}'$ admits the d\'evissage
$$
1\to G^{0}(C)\to J'_{C}\to J'_{\widetilde{C}}\to 1.
$$

For the compactification of $J'$ to $\overline{P}'$, since $\Phi$ is an \'etale Galois covering, the compactification is modeled on that of $\jac_{C}$ in $\overline{P}_{C}$.  
We define $\overline{P}''$ by the Cartesian diagram
$$
\begin{tikzcd}
\overline{P}''\arrow[r, "\Phi_{2}"]\arrow[d]& \overline{P}_{C}\arrow[d]\\
J'_{\widetilde{C}}\arrow[r]&\jac_{\widetilde{C}}\end{tikzcd}
$$
Then $\overline{P}''\to \overline{P}_{C}$ is a subcovering of $\Phi$, and we have factorization
$$
\Phi:\overline{P}'\xrightarrow{\Phi_{1}} \overline{P}''\xrightarrow{\Phi_{2}} \overline{P}_{C}. 
$$
The fibration $\overline{P}''\to J'_{\widetilde{C}}$ has fibre homeomorphic to $\Lambda^{0}(C)\backslash \xx^{0}(C)$. By abuse of notation, we identify this quotient as the fiber over the identity element $e'\in J'_{\widetilde{C}}$. 
We define $X'$ by the cartesian diagram
$$
\begin{tikzcd}
X'\arrow[r]\arrow[d]& \overline{P}'\arrow[d, "\Phi_{1}"]\\
\Lambda^{0}(C)\backslash \xx^{0}(C)\arrow[r]&\overline{P}''\end{tikzcd}
$$
By construction, ${J}_{C}'$ is dense open in $\overline{P}''$, and its subgroup $G^{0}(C)$ acts on $\Lambda^{0}(C)\backslash \xx^{0}(C)$ with dense open orbit. The covering $\Phi_{1}:  \overline{P}'\to \overline{P}''$ restricts to the trivial $\ker(\phi^{*}|_{\ct})$-covering ${J}'\to {J}'_{C}$. Hence ${J}'$ is isomorphic to $\ker(\phi^{*}|_{\ct})\times {J}'_{C}$, and its subgroup $\ker(\phi^{*}|_{\ct})\times G^{0}(C)$ acts on $X'$ with dense open orbit. 
As a consequence of proposition \ref{cj fibration general}, we obtain:

\begin{prop}\label{describe finite cj}
With the above notations, we have a universal homeomorphism
\begin{equation*}
\big[X'\times J'\big]\Big/\big(\ker(\phi^{*}|_{\ct})\times G^{0}(C)\big)\to \overline{P}'.
\end{equation*}
In particular, we have a fibration
$\overline{P}'\to J'_{\widetilde{C}}$
with fiber $X'$. 
\end{prop}

\subsubsection{Applications to the miniversal equivariant deformation}

We apply results from \S\ref{recall finite covering} to the family $f_{n}:\overline{\cp}_{n}\to \cb_{n}, (n,p)=1$.
We first briefly recall their construction from \cite{chen decomposition}. 
Let $A=\co[\gamma]$ and $\widetilde{A}$ the normalization of $A$ in $E=\mathrm{Frac}(A)$.
By the general theory of finite abelian coverings as explained in theorem 4.2 of \cite{chen decomposition}, a finite abelian covering $C'\to C_{\gamma}$ is  uniquely determined by a finite formal sum $\ct_{C'/C_{\gamma}}$ of torsion subgroups of $\jac_{C_{\gamma}}$. We calculate that $\pic^{0}_{C_\gamma}\cong \widetilde{A}^{\times}/A^{\times}$. It is product of a unipotent subgroup and a split maximal torus $\rS_{\gamma}=\bbg_{m}^{r}/\bbg_{m}$. 
Hence the torsion subgroup of $\pic_{C_\gamma}$ coincides with that of $\rS_{\gamma}$. 
For $n\in \bn, (n,p)=1$, we denote by $\psi_{n}: C_{n}\to C_{\gamma}$ the finite abelian covering corresponding to the torsion subgroup $\pic^{0}_{C_\gamma}[n]=\rS_{\gamma}[n]=\mu_{n}^{r}/\mu_{n}$.
The Galois group of the covering is naturally identified with the dual 
$$
\Hom\big(\rS_{\gamma}[n], {\overline{\mathbf{Q}}}_{\ell}^{\times}\big)=
\Hom\big(\,{\textstyle\frac{1}{n}}X_{*}(\rS_{\gamma})/X_{*}(\rS_{\gamma}), \bq/\bz\big)
\cong \Lambda^{0}/n.
$$
Similar construction works for the compactified Jacobian $\overline{P}_{C_{\gamma}}$.
By the autoduality of the compactified Jacobian \cite{arinkin 1}, \cite{arinkin 2}, we have
$$
\pic_{\overline{P}_{C_{\gamma}}/k}=\pic_{C_{\gamma}/k}. 
$$
The same torsion subgroup $\rS_{\gamma}[n]$ determines a $\Lambda_{0}/n$-covering $\overline{P}_{n}\to \overline{P}_{C_{\gamma}}$.
Moreover, we can find an open subscheme $\cb^{\circ}\subset \cb$, and quasi-finite \'etale coverings $\varpi_{n}:\cb_{n}\to \cb^{\circ}$ such that $\pic^{0}_{C_\gamma}[n]$ extends to a constant finite sub group scheme $\ct_{n}$ of $\pic_{\cc/\cb}\times_{\cb}\cb_{n}$, which then determines a $\Lambda^{0}/n$-covering 
$$
\Phi_{n}: \overline{\cp}_{n}\to \overline{\cp}\times_{\cb}\cb_{n}.
$$ 
Its composite with the structure morphism to $\cb_{n}$ gives the family $f_{n}:\overline{\cp}_{n}\to \cb_{n}$.

The finite sub group scheme $\ct_{n}$ of $\pic_{\cc/\cb}\times_{\cb}\cb_{n}$ can be described explicitly, this has been done in \cite{chen decomposition}, \S4.3.   
The fiber of $\ct_{n}$ at $0_{n}$ is isomorphic to
\begin{equation}\label{h1 curve copy}
H^{1}(C_\gamma,\bz/n)=J_{C_\gamma}[n]\otimes \mu_{n}^{-1}=\rS_{\gamma}[n]\otimes \mu_{n}^{-1}={\textstyle\frac{1}{n}}X_{*}(\rS_{\gamma})/X_{*}(\rS_{\gamma}),
\end{equation} 
For $i=1,\cdots, r$, let $\ep_{i}\in X_{*}(\rS_{\gamma})=X_{*}(\bbg_{m}^{r}/\bbg_{m})$ be the cocharacter sending $\bbg_{m}$ identically to the $i$-th factor of $\rS_{\gamma}$. 
Let $c_{i}\in H^{1}(C_\gamma,\bz)$ be the element such that $c_{i}\mod n$ corresponds to $\frac{1}{n}\ep_{i}$ under the isomorphism (\ref{h1 curve copy}) for all $n\in \bn, (n,p)=1$.
They extend to global sections of $\varprojlim\ct_{n}$ and form a set of generators; we denote the extension by the same notation $c_{i}$.   
In \cite{chen decomposition}, we get an expression for $c_{i}$ in terms of a set of distinguished basis of the Milnor fiber of $C_{\gamma}$: 
Take a generic point $s\in B_{\ep}(0)\cap \cb_{\delta_{\gamma}}$, the fiber $\cc_{s}$ is rational projective with $\delta_{\gamma}$ nodes.  Let $U_{c}\subset \cc$ be a sufficiently small open ball around the unique singularity $c$ of $C_{\gamma}$. 
Each of the irreducible component $\spf(\co[\gamma_{i}])$, $i=1, \cdots,r$, of $\spf(\co[\gamma])$ deforms along $\cb_{\delta_{\gamma}}$ into an irreducible component $\cc_{s,i}^{\circ}$ of $U_{c}\cap \cc_{s}$.
Let $z_{1}, \cdots, z_{\delta_{\gamma}}$ be the singularities of $\cc_{s}$, let $\alpha_{1}, \cdots, \alpha_{\delta_{\gamma}}\in H^{1}(\cc_{\bar{\xi}}, \bz)$ be the corresponding vanishing cycle. 
They form a set of distinguished basis for $H^{1}(\cc_{\bar{\xi}}, \bz)$.
For $i=1, \cdots, r$, let $J_{i}\subset \{1,\cdots, \delta_{\gamma}\}$ be the index set of the singularities lying on $\cc_{s, i}^{\circ}$, let $J_{i}^{\circ}\subset J_{i}$ be the subset for the singularities which lie also on another branch. 
By proposition 4.16 of \cite{chen decomposition}, we have
\begin{equation}\label{c express}
c_{i}=\sum_{j\in J_{i}^{\circ}}(\pm) \alpha_{j}, \quad i=1,\cdots, r.
\end{equation}
Moreover, for any $i\neq i'$ and $j\in J_{i}^{\circ}\cap J_{i'}^{\circ}$, the sign of $\alpha_{j}$ in the expression of $c_{i}$ and $c_{i'}$ are opposite. 
In particular, we have $\sum_{i=1}^{r}c_{i}=0$.

Let $x_{n}\in \cb_{n}$ be a geometric point lying over $x\in B_{\ep}(0)\cap \Delta$, let $\{z_{j}\}_{j\in J}$ be the set of singularities of $\cc_{x}$. 
We can use results from \S\ref{recall finite covering} to describe the finite abelian covering $\Phi_{n, x_{n}}:\overline{\cp}_{n,x_{n}}\to \overline{P}_{\cc_{x}}$ explicitly. 
We have a d\'evissage
$$
1\to G^{0}(\cc_{x})\to \jac_{\cc_{x}}\to \jac_{\widetilde{\cc}_{x}}\to 1,
$$
and a fibration
$$
\overline{P}_{\cc_{x}}\to \jac_{\widetilde{\cc}_{x}}
$$
with fibre homeomorphic to $\Lambda(\cc_{x})^{0}\backslash \xx(\cc_{x})^{0}$.
Let $\phi_{x}:\widetilde{\cc}_{x}\to \cc_{x}$ be the normalization, let
$\phi_{x}^{*}:\pic_{\cc_{x}}\to \pic_{\widetilde{\cc}_{x}}$ be the pull-back of invertible sheaves. 
Let $J'_{\widetilde{\cc}_{x}}$ be the finite abelian covering of $\jac_{\widetilde{\cc}_{x}}$ associated to the finite group scheme $\phi^{*}_{x}(\ct_{n, x})^{\red}$. 
Let $J'$ be the inverse image of $\jac_{\cc_{x}}$ for the finite abelian covering $\Phi_{n, x_{n}}:\overline{\cp}_{n,x_{n}}\to \overline{P}_{\cc_{x}}$, it is a trivial $\ker(\phi_{x}^{*}|_{\ct_{n,x_{n}}})$-covering of a finite abelian covering $J_{\cc_{x}}'$ of $\jac_{\cc_{x}}$, and $J_{\cc_{x}}'$ admits a d\'evissage
$$
1\to G^{0}(\cc_{x})\to J_{\cc_{x}}'\to J_{\widetilde{\cc}_{x}}'\to 1. 
$$
We define $\overline{P}''$ by the Cartesian diagram
$$
\begin{tikzcd}
\overline{P}''\arrow[r, "\Phi_{2}"]\arrow[d]& \overline{P}_{\cc_{x}}\arrow[d]\\
J'_{\widetilde{\cc}_{x}}\arrow[r]&\jac_{\widetilde{\cc}_{x}}\end{tikzcd}
$$
Then $\overline{P}''\to \overline{P}_{\cc_{x}}$ is a subcovering of $\Phi_{n, x_{n}}:\overline{\cp}_{n,x_{n}}\to \overline{P}_{\cc_{x}}$, and we have factorization
$$
\Phi_{n,x_{n}}:\overline{\cp}_{n,x_{n}}\xrightarrow{\Phi_{1}} \overline{P}''\xrightarrow{\Phi_{2}} \overline{P}_{\cc_{x}}. 
$$
The fibration $\overline{P}''\to J'_{\widetilde{\cc}_{x}}$ has fibre homeomorphic to $\Lambda^{0}(\cc_{x})\backslash \xx^{0}(\cc_{x})$. By abuse of notation, we identify this quotient as the fiber over the identity element $e'\in J'_{\widetilde{\cc}_{x}}$. 
We define $X'$ by the cartesian diagram
\begin{equation}\label{define x'}
\begin{tikzcd}
X'\arrow[r]\arrow[d]& \overline{\cp}_{n,x_{n}}\arrow[d, "\Phi_{1}"]\\
\Lambda^{0}(\cc_{x})\backslash \xx^{0}(\cc_{x})\arrow[r]&\overline{P}''\end{tikzcd}
\end{equation}
By construction, ${J}_{\cc_{x}}'$ is dense open in $\overline{P}''$, and its subgroup $G^{0}(\cc_{x})$ acts on $\Lambda^{0}(\cc_{x})\backslash \xx^{0}(\cc_{x})$ with dense open orbit. The covering $\Phi_{1}:  \overline{\cp}_{n,x_{n}}\to \overline{P}''$ restricts to the trivial $\ker(\phi_{x}^{*}|_{\ct_{n,x_{n}}})$-covering ${J}'\to {J}'_{\cc_{x}}$. Hence ${J}'$ is isomorphic to $\ker(\phi_{x}^{*}|_{\ct_{n,x_{n}}})\times {J}'_{\cc_{x}}$, and its subgroup $\ker(\phi_{x}^{*}|_{\ct_{n,x_{n}}})\times G^{0}(\cc_{x})$ acts on $X'$ with dense open orbit. 
By proposition \ref{describe finite cj}, we get

\begin{prop}\label{describe finite cj alt}
With the above notations, we have a universal homeomorphism
\begin{equation*}
\big[X'\times J'\big]\Big/\big(\ker(\phi_{x}^{*}|_{\ct_{n,x_{n}}})\times G^{0}(\cc_{x})\big)\to \overline{\cp}_{n, x_{n}}.
\end{equation*}
In particular, we have a fibration
$\overline{\cp}_{n, x_{n}}\to J'_{\widetilde{\cc}_{x}}$
with fiber homeomorphic to $X'$. 
\end{prop}

Recall that in \cite{chen decomposition} we have defined for each non-trivial partition $\{1, \cdots,r\}=\bigsqcup_{j=1}^{l}I_{j}$ a strata $\cs_{I_{\bullet}}\subset \cb$. 
For such a partition, we set $c_{I_{j}}=\sum_{i\in I_{j}}c_{i}$, and we define $\cs_{I_{\bullet}}$ by the condition that $x\in \cs_{I_{\bullet}}$ if and only if all the classes $c_{I_{1}}, \cdots, c_{I_{l}}$ vanish on the fiber $\cc_{x}$.
Let $\cs_{I_{\bullet}}^{\circ}$ be the open subscheme of $\cs_{I_{\bullet}}$ over which the fibers of $\pi:\cc\to \cb$ have only ordinary double points, it is one of the strict $\delta$-strata. 
Let $\cs_{I_{\bullet}, n}=\cs_{I_{\bullet}}\times_{\cb}\cb_{n}$ and $\cs_{I_{\bullet},n}^{\circ}=\cs_{I_{\bullet},n}\times_{\cs_{I_{\bullet}}}\cs_{I_{\bullet}}^{\circ}$.

\begin{prop}\label{multiple locus}

The finite abelian covering $X'\to \Lambda^{0}(\cc_{x})\backslash \xx^{0}(\cc)$ is non-trivial if and only if $x\in \cs_{I_{\bullet}}$ for some non-trivial partition $I_{\bullet}\vdash \{1,\cdots, r\}$. Moreover, for $x\in \cs_{I_{\bullet}}^{\circ}$, the Galois group is isomorphic to $(\bz/n)^{\oplus(l-1)}$, where $l$ is the length of the partition $I_{\bullet}$. 

\end{prop}

\begin{proof}

With the Cartesian diagram (\ref{define x'}), the covering $X'\to \Lambda^{0}(\cc_{x})\backslash \xx^{0}(\cc)$ is of the same Galois group as that of $\overline{\cp}_{n,x_{n}}\to \overline{P}''$, which is isomorphic to $\ker(\phi_{x}^{*}|_{\ct_{n,x_{n}}})$.
By theorem 4.3 of \cite{chen decomposition}, $\ker(\phi_{x}^{*}|_{\ct_{n,x_{n}}})$ is non-trivial if and only if the fiber $\overline{\cp}_{n,x_{n}}$ has multiple irreducible components.
By theorem 4.27 of \cite{chen decomposition}, this happens if and only if $x_{n}\in \cs_{I_{\bullet}, n}$ for some non-trivial partition $I_{\bullet}\vdash \{1,\cdots, r\}$, i.e. $x\in \cs_{I_{\bullet}}$. 
Moreover, for $x\in \cs_{I_{\bullet}}^{\circ}$, $\ker(\phi_{x}^{*}|_{\ct_{n,x_{n}}})$ is generated by $c_{I_{1}} \mod n,\cdots, c_{I_{l}} \mod n$, and they satisfy a unique non-trivial linear equation $c_{I_{1}}+\cdots+c_{I_{l}}=0$ by the explicite expression (\ref{c express}), so $\ker(\phi_{x}^{*}|_{\ct_{n,x_{n}}})\cong (\bz/n)^{\oplus (l-1)}$, whence the second assertion.

\end{proof}

\subsubsection{The Morse groups for the primitive part}\label{micro primitive}

As explained in the beginning of the section, the primitive part is actually a limit (\ref{primitive as limit}).
By construction, $\varpi_{n}:\cb_{n}\to \cb^{\circ}$ is \'etale and $0_{n}\in \cb_{n}$ is the point lying over $0\in \cb_{n}$. 
Hence the iterated projections for the base $\cb$ can be carried over to $\cb_{n}$ and we can make microlocal analysis of $j_{n, !*}R^{i}f^{\sm}_{n, *}\bz/\ell^{m}$ with the same process as explained in \S\ref{cj microlocal}. 
We will then examine how the Morse groups for $(j_{n})_{!*}Rf_{n, *}^{\sm}\bz/\ell^{m}$ changes as we go up in the tower $f_{n}$. Take limit, we get the Morse groups for the primitive part. 
This process can be simplified: Choose compatible open neighborhoods $U_{n}$ of $0_{n}$ in $\cb_{n}$ such that $\varpi_{n}$ restricts to an analytic isomorphism from $U_{n}$ to $U_{0}$,
we can then carry over the sheaves $j_{n, !*}R^{i}f^{\sm}_{n, *}\bz/\ell^{m}$ from $U_{n}$ to $U_{0}$, and perform the microlocal analysis over $U_{0}$.
We keep the same notations from \S\ref{cj microlocal}.

\begin{thm}\label{morse primitive}

Let $x\in F^{k}\cap\cb_{\beta}$ be a critical point of $\varphi_{k+1}$ with critical value $v$, let $\lambda$ be the Morse index of the restriction of $\varphi_{k+1}$ to $F^{k}\cap\cb_{\beta}$.
Let $a$ be a sufficiently small positive real number.

\begin{enumerate}[topsep=0pt, itemsep=0pt, label=$(\arabic*)$]

\item

If $\cb_{\beta}$ is not one of the strict $\delta$-strata, then
$$
\Big(\varprojlim_{m}\varinjlim_{n}H^{*}\big(F^{k}_{<v+a}, F^{k}_{<v-a}; \, j_{n, !*}R^{i}f^{\sm}_{n,*}\bz/\ell^{m}\big)\Big)\otimes_{\bz_{\ell}}\ql=0, \quad \text{for all } i. 
$$

\item 

If $\cb_{\beta}$ is one of the strict $\delta$-strata $\cs^{\circ}_{I_{\bullet}}$, $I_{\bullet}$ a non-trivial partition of $\{1,\cdots, r\}$, then
$$
\Big(\varprojlim_{m}\varinjlim_{n}H^{*}\big(F^{k}_{<v+a}, F^{k}_{<v-a}; \, j_{n, !*}R^{i}f^{\sm}_{n,*}\bz/\ell^{m}\big)\Big)\otimes_{\bz_{\ell}}\ql=0, \quad \text{for all } i. 
$$

\item

If $\cb_{\beta}$ is a strict $\delta$-strata which is not of the form $\cs^{\circ}_{I_{\bullet}}$, $I_{\bullet}$ a non-trivial partition of $\{1,\cdots, r\}$, then
$$
\Big(\varprojlim_{m}\varinjlim_{n} H^{\lambda+\delta}\big(F^{k}_{<v+a}, F^{k}_{<v-a}; \,j_{n, !*}R^{i}f^{\sm}_{n,*}\bz/\ell^{m}\big)\Big)\otimes_{\bz_{\ell}}\ql=
\BigWedge^{i-\delta}\phi_{x}^{*}(\cf_{x}/\ct_{x, \ell}),$$
where $\ct_{x, \ell}$ is the $\ell$-primary component of $\varprojlim_{n} \ct_{n,x}$, and all the other Morse groups vanish. 
\end{enumerate}

\end{thm}

\begin{proof}

Let $N_{x}$ be a transversal slice to $\cb_{\beta}$ at $x$ in $\cb$, let $\rp:(N_{x}, x)\to (\ba^{1}, 0)$ be a generic projection. Let $N_{x, n}=\varpi_{n}^{-1}(N_{x})$, let $\rp_{n}=\rp\circ \varpi_{n}:N_{x,n}\to \ba^{1}$, let $\bar{\eta}_{0}$ be a geometric generic point of $\ba^{1}_{\{0\}}$, let $R\Phi_{n,\bar{\eta}_{0}}$ be the vanishing cycle functor with respect to $\rp_{n}$.
By the stratified Morse theory of Goresky and MacPherson \cite{gm morse}, up to a shift by the Morse index, the Morse group in question is exactly the limit of the vanishing cycle
$$
\Big(\varprojlim_{m}\varinjlim_{n}(R^{q}\Phi_{n, \bar{\eta}_{0}}(j_{n, !*}R^{i}f_{n, *}^{\sm}\bz/\ell^{m}))_{x_{n}}\Big)\otimes_{\bz_{\ell}}\ql.
$$

$(1)$ If $\cb_{\beta}$ is not one of the strict $\delta$-strata, by theorem 3.3 of \cite{chen root valuation}, we have 
$$
(R^{q}\Phi_{\bar{\eta}_{0}}(j_{!*}R^{i}f^{\sm}_{*}\ql))_{x}=0, \quad \text{for all }q\text{ and } i.
$$
By construction, $f_{n}:\overline{\cp}_{n}\to \cb_{n}$ is the composite of $\Phi_{n}: \overline{\cp}_{n}\to \overline{\cp}\times_{\cb}\cb_{n}$ with the base change of $f$ to $\cb_{n}$. Since the restriction of $\Phi_{n}$ to $\cb_{n}^{\sm}$ is an isogeny of Abelian scheme, and the morphism $\varpi_{n}:\cb_{n}\to \cb^{\circ}$ is \'etale, we have
$$
(R^{q}\Phi_{n,\bar{\eta}_{0}}(j_{n, !*}R^{i}f_{n, *}^{\sm}\ql))_{x_{n}}=0, \quad \text{for all }q\text{ and } i.
$$
This implies that for all $n$ the vanishing cycle
$$
(R^{q}\Phi_{n, \bar{\eta}_{0}}(j_{n, !*}R^{i}f_{n, *}^{\sm}\bz/\ell^{m}))_{x_{n}}
$$
is a fixed torsion group when $m$ is sufficiently large.
Hence
$$
\varprojlim_{m}\varinjlim_{n}(R^{q}\Phi_{n, \bar{\eta}_{0}}(j_{n, !*}R^{i}f_{n, *}^{\sm}\bz/\ell^{m}))_{x_{n}}
$$
remains a torsion group, its tensor product with $\ql$ then vanish.

$(2)$
If $\cb_{\beta}$ is one of the strict $\delta$-strata $\cb_{\delta}^{\circ}$, the restriction to $N_{x}$ of the family $\pi:\cc\to \cb$ is analytically isomorphic to the product of the miniversal deformations of the singularities $\widehat{\cc}_{x, z_{j}}, j\in J$. 
Let $C_{0}$ be a rational projective irreducible curve with $\delta$ ordinary double points, then $\overline{\cp}_{x}$ admits a fibration $\overline{\jac}_{\cc_{x}}\to \jac_{\widetilde{\cc}_{x}}$ with fiber homeomorphic to $\overline{P}_{C_{0}}$. For simplicity, we will denote it by
\begin{equation}\label{cj factorize}
\begin{tikzcd}\overline{P}_{C_{0}} \arrow[r] &\overline{\jac}_{\cc_{x}}\arrow[d]\\
& \jac_{\widetilde{\cc}_{x}}.
\end{tikzcd}
\end{equation}
Let $C_{xy}$ be the irreducible projective rational curve over $\bc$ with a unique ordinary double point $0_{xy}$ as singularity, let $\overline{P}_{xy}$ be its compactified Jacobian, then 
\begin{equation}\label{factor delta node}
\overline{P}_{C_{0}}=\underbrace{\overline{P}_{xy}\times\cdots\times\overline{P}_{xy}}_{\delta}\cong \underbrace{\overline{C}_{xy}\times\cdots\times\overline{C}_{xy}}_{\delta}.
\end{equation}   
Let $\psi:(C,C_{0})\to (B=\ba^{\delta}, 0)$ be a miniversal deformation of $C_{0}$, and let $\varphi:\overline{\pic}_{C/B}\to B$ be its relative compactified Jacobian. Let $\jmath:B^{\sm}\to B$ be the inclusion of the open subscheme over which $\varphi$ is smooth, and let $\varphi^{\sm}$ be the restriction of $\varphi$ to the inverse image of $B^{\sm}$. 
Via the fibration (\ref{cj factorize}), the family $\varphi$ serves as a local model for the restriction $f|_{N_{x}}$, up to the abelian factor $\jac_{\widetilde{\cc}_{x}}$. 
By proposition 7.4.10 of \cite{ngo}, we obtain a factorization 
\begin{equation*}
j_{!*}Rf_{*}^{\sm}\ql
=H^{*}\big(\jac_{\widetilde{\cc}_{x}}, \ql\big)\otimes \jmath_{!*}R\varphi_{*}^{\sm}\ql.  
\end{equation*}
As explained in the proof of theorem \ref{morse group}, the composition $\rp\circ f|_{N_{x}}$ is smooth, except at the point $q_{0}\in \overline{P}_{C_{0}}$ which correspond to $(0_{xy}, \cdots, 0_{xy})\in C_{xy}\times \cdots\times C_{xy}$ under the map (\ref{factor delta node}), and  locally at $q_{0}$ the mapping $\rp\circ f|_{N_{x}}$ can be described as
\begin{equation}\label{local model}
\bc^{2\delta}\to \bc,\quad (x_{1}, y_{1}, \cdots,x_{\delta},y_{\delta})\mapsto \sum_{i=1}^{\delta}x_{i}y_{i},
\end{equation}
which is the miniversal deformation of an ordinary quadratic singularity.
Let $p:(B,0)\to( \ba^{1}, 0)$ be a generic projection, let $R\Phi_{p, \bar{\eta}_{0}}$ be the vanishing cycle functor for the projections $p$. 
With the local model (\ref{local model}), we obtain that
$$
\big(R\Phi_{p, \bar{\eta}_{0}}(\jmath_{!*}R\varphi^{\sm}_{*}\bz/\ell^{m})\big)_{x}=\big(R\Phi_{p, \bar{\eta}_{0}}R\varphi_{*}\bz/\ell^{m})_{x}=\bz/\ell^{m}[-2\delta+1]
$$ 
is generated by a vanishing cycle $\vartheta$.

By proposition \ref{describe finite cj alt}, we have a finite abelian covering $\overline{P}'_{C_{0}}$ of $\overline{P}_{C_{0}}$ determined by $\ker(\phi_{x}^{*}|_{\ct_{n,x}})$, and a finite abelian covering  $J'_{\widetilde{\cc}_{x}}$ of $\jac_{\widetilde{\cc}_{x}}$ determined by the finite group scheme $(\phi^{*}_{x}\ct_{n,x})^{\red}$, such that $\overline{\cp}_{n,x_{n}}$ admits a fibration 
\begin{equation}\label{cj factorized 2}
\begin{tikzcd}\overline{P}'_{C_{0}} \arrow[r] &\overline{\cp}_{n,x_{n}}\arrow[d]\\
& J'_{\widetilde{\cc}_{x}}.
\end{tikzcd}
\end{equation}
By proposition \ref{multiple locus}, the finite abelian covering $\overline{P}'_{C_{0}}\to \overline{P}_{C_{0}}$ is non-trivial if and only if $x\in \cs_{I}^{\circ}$ for a non-trivial partition $I_{\bullet}\vdash\{1,\cdots, r\}$. In this case, the Galois group is isomorphic to $(\bz/n)^{\oplus (l-1)}$.
Moreover, this finite abelian covering extends to a $(\bz/n)^{\oplus (l-1)}$-covering $\Psi_{n}:\overline{P}_{n}\to \overline{\pic}_{C/B}\times_{B}B_{n}$, where $\omega_{n}:B_{n}\to B^{\circ}$ is a quasi-finite \'etale morphism. 
Let $\varphi_{n}:\overline{P}_{n}\to B_{n}$ be the composite of $\Psi_{n}$ with the structure morphism to $B_{n}$, it is a miniversal $(\bz/n)^{(l-1)}$-equivariant deformation of $\overline{P}_{C_{0}}'$.
Via the fibration (\ref{cj factorized 2}), the family $\varphi_{n}$ serves as a local model for the restriction of $f_{n}:\overline{\cp}_{n}\to \cb_{n}$ to the inverse image of $N_{x,n}=N_{x}\times_{\cb}\cb_{n}$, up to the abelian factor $J'_{\widetilde{\cc}_{x}}$.
Let $\jmath_{n}:B_{n}^{\sm}\to B_{n}$ be the inclusion. By proposition 7.4.10 of \cite{ngo}, we obtain a factorization  
\begin{equation}
(j_{n})_{!*}Rf_{n,*}^{\sm}\ql
=H^{*}\big(J'_{\widetilde{\cc}_{x}}, \ql\big)\otimes (\jmath_{n})_{!*}R\varphi_{n,*}^{\sm}\ql,  
\label{ic factorize}
\end{equation}

Let $p_{n}$ be the composite $B_{n}\to B\to \ba^{1}$, and let $R\Phi_{p_{n}, \bar{\eta}_{0}}$ be its associated vanishing cycle functor. 
Since $\Psi_{n}$ is a finite abelian covering of Galois group $(\bz/n)^{\oplus (l-1)}$, the inverse image $\Psi_{n}^{-1}(q_{0})$ consists of $n^{(l-1)}$ points $q_{1}, \cdots, q_{n^{(l-1)}}$, the composite $p_{n}\circ \varphi_{n}$ is smooth except at these points, and the singularity there is an ordinary double point as before. Let $\vartheta_{i}$ be the vanishing cycle at $q_{i}$, $i=1, \cdots, n^{(l-1)}$, they are pull-back of the vanishing cycle $\vartheta$ via $\Psi_{n}$, and they are homologous to each other, so
\begin{align*}
\big(R\Phi_{p_{n}, \bar{\eta}_{0}}(\jmath_{n, !*}R\varphi^{\sm}_{n,*}\bz/\ell^{m})\big)_{x_{n}}&=\big(R\Phi_{p_{n}, \bar{\eta}_{0}}(R\varphi_{n,*}\bz/\ell^{m})\big)_{x_{n}}
\\
&=\bz/\ell^{m}\cdot (\vartheta_{1}+\cdots+\vartheta_{n^{(l-1)}})[-2\delta+1]
\\
&=\bz/\ell^{m}\cdot (n^{(l-1)}\vartheta)[-2\delta+1],
\end{align*}
whence
$$
\varprojlim_{m}\varinjlim_{n}\big(R\Phi_{p_{n}, \bar{\eta}_{0}}(\jmath_{n, !*}R\varphi^{\sm}_{n,*}\bz/\ell^{m})\big)_{x_{n}}=0.
$$
By the factorization (\ref{ic factorize}), we have
$$
\varprojlim_{m}\varinjlim_{n}\big(R\Phi_{n, \bar{\eta}_{0}}(j_{n})_{!*}Rf_{n,*}^{\sm}\bz/\ell^{m}\big)_{x_{n}}
\subset H^{*}\big(J'_{\widetilde{\cc}_{x}}, \ql\big)\otimes\varprojlim_{m}\varinjlim_{n}\big(R\Phi_{p_{n}, \bar{\eta}_{0}}(\jmath_{n, !*}R\varphi^{\sm}_{n,*}\bz/\ell^{m})\big)_{x_{n}}=0.
$$
This finishes the proof of the case $(2)$ of the theorem.


In case $(3)$, the group $\ker(\phi_{x}^{*}|_{\ct_{n, x}})$ is trivial, so $\overline{P}'_{C_{0}}=\overline{P}_{C_{0}}$ in the fibration (\ref{cj factorized 2}). 
The problem is reduced to studying how the isogeny $J_{\widetilde{\cc}_{x}}'\to \jac_{\widetilde{\cc}_{x}}$ varies as we go up in the tower $f_{n}$.
Recall that $J'_{\widetilde{\cc}_{x}}$ is the finite abelian covering of $\jac_{\widetilde{\cc}_{x}}$ associated to the finite group scheme $\phi_{x}^{*}(\ct_{n, x})$, i.e. the quotient of 
$
J'_{\widetilde{\cc}_{x}}$ by $\phi_{x}^{*}(\ct_{n, x})$ is isomorphic to $\jac_{\widetilde{\cc}_{x}},
$
so we have the exact sequence
$$
0\to H_{1}\big(J'_{\widetilde{\cc}_{x}}, \bz\big)\to H_{1}\big(\jac_{\widetilde{\cc}_{x}}, \bz\big) \to  \phi_{x}^{*}(\ct_{n, x})\to 0.
$$
This implies that
\begin{equation*}
\Big(\varprojlim_{m}\varinjlim_{n}H^{*}(J'_{\widetilde{\cc}_{x}}, \bz/\ell^{m})\Big)\otimes_{\bz_{\ell}}\ql=\BigWedge^{\bullet}\phi_{x}^{*}(\cf_{x}/\ct_{x, \ell}).
\end{equation*}
With the factorization (\ref{ic factorize}), we obtain
$$
\Big(\varprojlim_{m}\varinjlim_{n}\big(R^{\delta-1}\Phi_{n, \bar{\eta}_{0}}(j_{n, !*}R^{i}f^{\sm}_{n,*}\bz/\ell^{m})\big)_{x_{n}}\Big)\otimes_{\bz_{\ell}}\ql=\BigWedge^{i-\delta}\phi_{x}^{*}(\cf_{x}/\ct_{x, \ell}),
$$
and all the other vanishing cycles vanish. This finishes the proof of the case $(3)$ of the theorem.

\end{proof}

\subsection{Microlocal analysis for the $\xi$-stable quotient}\label{micro xi-stable}

Consider the miniversal deformation $\ppi:(\kxx, X_{\gamma})\to (\kbb, 0)$ of the spectral curve $X_{\gamma}$, and its relative $\xi$-stable compactified Jacobian $\ff^{\xi}:\overline{\kjj}^{\,\xi}\to \kbb$.
By theorem \ref{cj support} and proposition \ref{cj springer stable}, we have
\begin{equation}\label{quotient as fiber}
H^{i}\big(\xx_{\gamma}^{0,\, \xi}/\rS, \ql\big)=H^{i}\big(\overline{J}{}^{\,\xi}_{X_{\gamma}}, \ql\big)=\bigoplus_{i'=0}^{i} \ch^{i-i'}(\jj_{!*}R^{i'}\ff^{\xi,\sm}_{*}\ql)_{0}, \quad i=0, \cdots, 2p_{a}(X_{\gamma}).
\end{equation}
We will make microlocal analysis of the complex $\jj_{!*}R^{i}\ff^{\xi,\sm}_{*}\ql=\jj_{!*}\BigWedge^{i}R^{1}\ff^{\xi,\sm}_{*}\ql$. 
For simplicity, we denote $\cg=R^{1}\ppi^{\sm}_{*}\ql=R^{1}\ff^{\xi, \sm}_{*}\ql$ and $\cg^{i}=\BigWedge^{i}\cg$ for $i=0, \cdots, 2p_{a}(X_{\gamma})$. 
By proposition \ref{quotient globalized}, for any sufficient small open neighborhood $U$ of $0\in \kbb$, we have $\cg|_{U-\Delta}=\ce/\ce^{\perp}$.
As before, we assume that $\dim(\kbb)=\tau_{\gamma}>\delta_{\gamma}$, which excludes the case that $\gamma=\diag(\gamma_{1}, \gamma_{2})\in \ggl_{2}(F)$ such that $\val(\gamma_{1}-\gamma_{2})=1$.

Let $\big(\{\kxx_{\alpha}\}_{\alpha\in \overline{\Omega'}}, \{\kbb_{\beta}\}_{\beta\in \Omega'}\big)$ be the canonical Whitney stratification for the family $\ppi:\kxx\to \kbb$ (cf. \cite{mather stratification}, proposition 10.1). 
Then the collection of the strict $\delta$-strata is part of the Whitney stratification $\{\kbb_{\beta}\}_{\beta\in \Omega'}$.
By construction, there exists sufficiently small open neighborhoods 
$0\in V\subset \cc$ and $0\in V'\subset \kxx$, such that the restriction to $V$ of $\pi:(\cc, C_{\gamma})\to (\cb, 0)$ and the restriction to $V'$ of $\ppi:(\kxx, X_{\gamma})\to (\kbb, 0)$  are miniversal deformations of the singularity $\spf(\co[\gamma])$. 
Hence we can find sufficiently small open neighborhoods $0\in U\subset \cb$ and $0\in U'\subset \kbb$, such that the restriction of the Whitney stratification $\{\cb_{\beta}\}_{\beta\in \Omega}$ and $\{\kbb_{\beta'}\}_{\beta'\in \Omega'}$ to $U$ and $U'$ respectively are isomorphic. 
Since we are performing microlocal analysis at the vicinity of $0$, we can transfer the iterated projections defined for $\cb$ in \S\ref{cj microlocal} to $\kbb$.
For simplicity, we keep the same notation, this should not cause confusion. 

We fix an embedding $(\kbb, 0)\to (\ba^{N}, 0)$. 
Let $S_{0}$ be the intersection of $\kbb$ with a generic linear subspace $H_{0}$ of $\ba^{N}$ of codimension $\tau_{\gamma}-(\delta_{\gamma}+1)$ such that it intersects transversally with the tangent cone at $0$ of all the $\delta$-strata. We then cut off the locus where its intersection with the union of all the $\delta$-strata is not transversal, to get a dense open subscheme $\kbb'\subset S_{0}$ containing $0$. 
Let $\ppi':\kxx'\to \kbb'$ be the restriction of the family $\ppi:\kxx\to \kbb$ to the inverse image of $\kbb'$, consider its relative $\xi$-stable compactified Jacobian
$$
(\ff')^{\xi}:\overline{\kjj'}^{\,\xi}=\overline{\pic}{}^{\,0,\,\xi}_{\kxx'/\kbb'}\to \kbb'.
$$
It is the restriction of the family $\ff^{\xi}:\overline{\kjj}^{\,\xi}\to \kbb$ to the inverse image of $\kbb'$.

\begin{prop}

The total space $\overline{\kjj'}^{\,\xi}$ is regular, and we have
\begin{equation*}
R(\ff')^{\xi}_{*}\ql=\bigoplus_{i=0}^{2p_{a}(X_{\gamma})} \jj'_{!*}R^{\,i}{(\ff')}^{\xi, \sm}_{*}\ql[-i],
\end{equation*}
where $\jj':(\kbb')^{\sm}\to \kbb'$ is the natural inclusion. In particular,
$$
H^{i}\big(\xx_{\gamma}^{0,\, \xi}/\rS, \ql\big)=H^{i}\big(\overline{J}{}^{\,\xi}_{X_{\gamma}}, \ql\big)=\bigoplus_{i'=0}^{i} \ch^{i-i'}(\jj'_{!*}R^{i'}(\ff')^{\xi,\sm}_{*}\ql)_{0}, \quad i=0, \cdots, 2p_{a}(X_{\gamma}).
$$
\end{prop}

\begin{proof}

The regularity of the total space $\overline{\kjj'}^{\,\xi}$ is a consequence of \cite{miglio 2}, theorem 1.9. 
The proof of the support theorem follows the same lines of reasoning for \cite{laumon lemme 2}, th\'eor\`eme 10.5.1. 
Indeed, the fibration $(\ff')^{\xi}:\overline{\kjj'}^{\,\xi}\to \kbb'$ is an abelian fibration in the sense of Ng\^o, and the inequality of Severi holds for this family by construction. Hence, all the ingredients required for the proof are present in the current setting.


\end{proof}

For simplicity, we keep the notation $\cg^{i}$ for the sheaf $R^{\,i}{(\ff')}^{\xi, \sm}_{*}\ql$.
We take a flag of generic linear subspaces
$$
\{0\}= H_{\delta_{\gamma}+1}\subsetneq H_{\delta_{\gamma}}\subsetneq \cdots \subsetneq H_{2}\subsetneq H_{1}\subsetneq H_{0}, \quad \mathrm{codim}_{H_{0}}(H_{k})=k \text{ for all }1\le k \le \delta_{\gamma}+1,
$$
for which two properties concerning the local polar varieties of the relevant Whitney strata, proposition 4.11 and 4.14 of \cite{chen root valuation}, hold. 
For $k=1, \cdots, \delta_{\gamma}$, let $\pi_{k}:H_{0}\to \bc^{k}$ be the projection with kernel $H_{k}$, let $t_{k}\in \bc^{k}$ be a geometric generic point sufficiently close to $0$. The projection $\pi_{1}$ induces a generic projection
$
p_{1}:\kbb'\cap B_{0}(\ep)\to \bc^{1}, 
$
and the projection $\pi_{k+1}$ induces an affine map 
$$
p_{k+1}:\kbb'\cap B_{0}(\ep)\cap\pi_{k}^{-1}(t_{k})\to \bc^{1}, \quad \text{for }k=1,\cdots, \delta_{\gamma}.
$$ 
The points $\{t_{k}\}_{k=1}^{\delta_{\gamma}}$ have been taken such that $\cb'\cap B_{0}(\ep)\cap \pi_{k+1}^{-1}(t_{k+1})$ coincides with the fiber of $p_{k+1}$ at a generic point in ${\rm Im}(p_{k+1})$. 
For simplicity, we denote 
$$F^{k}=\begin{cases} \kbb'\cap B_{0}(\ep),& \text{for }k=0,\\
\kbb'\cap B_{0}(\ep)\cap\pi_{k}^{-1}(t_{k}), &\text{for }k=1,\cdots, \delta_{\gamma},\end{cases}
$$ 
it is of dimension $\delta_{\gamma}+1-k$. 
Note that $F^{\delta_{\gamma}}$ is a generic line in $\kbb'\cap B_{0}(\epsilon)$, and a generic geometric point on it is the same as a generic geometric point of $\kbb'\cap B_{0}(\epsilon)$, we will denote it by $\bar{\xi}$.  
We then replace the projection $p_{k+1}$ by a stratified Morse function 
$$
\varphi_{k+1}:F^{k}\to \br^{1}, \quad k=0, \cdots, \delta_{\gamma},
$$ 
with respect to the stratification of $F^{k}$ induced from the  canonical Whitney stratification $\{\kbb_{\beta}\}_{\beta\in \Omega}$, which approximates sufficiently well to the function ${\rm Re}(p_{k+1})$. 
Without loss of generality, we can assume that the critical points of $\varphi_{k+1}$ have distinct critical values.
For $r\in \br$, let $F^{k}_{<r}=\varphi_{k+1}^{-1}(-\infty, r)$.
Let $v_{k}$ be the smallest critical value of $\varphi_{k+1}$, then
$$
H^{*}\big(F^{k}_{<v_{k}},  \jj'_{!*}\cg^{i} \big)= \begin{cases}
\cg^{i}_{\bar{\xi}}, &\text{ for }k=\delta_{\gamma},\\
H^{*}\big(F^{k+1}, \jj'_{!*}\cg^{i}\big), &\text{ for }k=\delta_{\gamma}-1,\cdots, 0,\end{cases}
$$
and the cohomology group $H^{*}\big(F^{k},  \jj'_{!*}\cg^{i} \big)$ can be built up from it and the Morse groups at the critical points $v$ of $\varphi_{k+1}$ via the long exact sequence
{\small
\begin{equation*}
\cdots \to H^{*-1}\big(F^{k}_{<v-a}, \jj_{!*}\cg^{i}\big) \xrightarrow{\partial} H^{*}\big(F^{k}_{<v+a}, F^{k}_{<v-a}; \jj_{!*}\cg^{i}\big)  \to  H^{*}\big(F^{k}_{<v+a},  \jj_{!*}\cg^{i}\big) \to H^{*}\big(F^{k}_{<v-a}, \jj_{!*}\cg^{i}\big) \to \cdots.
\end{equation*}
}

\noindent In this way, we can compute iteratively the cohomologies $H^{*}\big(F^{k},  \jj'_{!*}\cg^{i} \big)$, starting from $\cg^{i}_{\bar{\xi}}$ and arriving at $H^{*}\big(F^{0},  \jj'_{!*}\cg^{i} \big)=\ch^{*}(\jj_{!*}\cg^{i})_{0}$.

\begin{thm}\label{xi morse group}

Let $x\in F^{k}\cap\kbb_{\beta}$ be a critical point of $\varphi_{k+1}$ with critical value $v$, let $\lambda$ be the Morse index of the restriction of $\varphi_{k+1}$ to $F^{k}\cap\kbb_{\beta}$.
Let $a$ be a sufficiently small positive real number.

\begin{enumerate}[topsep=0pt, itemsep=0pt, label=$(\arabic*)$]

\item

If $\kbb_{\beta}$ is not one of the strict $\delta$-strata, then
$$
H^{*}\big(F^{k}_{<v+a}, F^{k}_{<v-a}; \jj'_{!*}\cg^{i}\big)=0,\quad \text{for all }i.
$$

\item If $\kbb_{\beta}$ is one of the strict $\delta$-strata $\kbb^{\circ}_{I_{\bullet}}$, $I_{\bullet}$ a non-trivial partition of $\{1,\cdots, r\}$, then
$$
H^{*}\big(F^{k}_{<v+a}, F^{k}_{<v-a}; \jj'_{!*}\cg^{i}\big)=0,\quad \text{for all }i.
$$

\item

If $\kbb_{\beta}$ is a strict $\delta$-strata not of the form $\kbb_{I_{\bullet}}^{\circ}$, $I_{\bullet}$ a non-trivial partition of $\{1,\cdots, r\}$, then the curve $\kxx_{x}$ is irreducible. Let ${\kxx}^{\flat}_{x}$ be its normalization and let 
$$
H^{*}(\jac_{{\kxx}^{\flat}_{x}}, \ql)=\bigoplus_{n=0}^{2(p_{a}(X_{\gamma})-\delta)}\boldsymbol{\Lambda}_{x}^{n}[-n],
$$ 
then for all $i$ we have
$$
H^{q}\big(F^{k}_{<v+a}, F^{k}_{<v-a}; \jj'_{!*}\cg^{i}\big)=
\begin{cases}
\boldsymbol{\Lambda}_{x}^{{i-\delta}}, & \text{ if }q=\lambda+\delta,
\\
0,&\text{otherwise}.
\end{cases}
$$

\end{enumerate}

\end{thm}

\begin{proof}

The proof follows the same lines of reasoning as proposition 3.8 and theorem 3.3 of \cite{chen root valuation}. Let $V$ be a transversal slice to $\kbb_{\beta}$ at $x$, let $p:V\to \ba^{1}$ be a generic projection sending $x$ to $0$. 
Let $\bar{\eta}_{0}$ be a geometric generic point of the strict Henselization $\ba^{1}_{\{0\}}$, let $R\Phi_{\bar{\eta}_0}$ be the vanishing cycle functor with respect to the projection $p$.
As explained in \cite{chen root valuation}, up to a shift by the Morse index $\lambda$, the Morse group is equal to the vanishing cycle $R^{\dim(V)-1}\Phi_{\bar{\eta}_0}\jj'_{!*}\cg^{i}$ at $x$.

Take a sufficiently small open neighborhood $\widetilde{U}$ of the unique singularity $0$ of $X_{\gamma}$ in $\kxx$, such that the restriction of the family $\ppi:\kxx\to \kbb$ to $\widetilde{U}$ is a miniversal deformation of the singularity $\spf(\co[\gamma])$, let $U\subset\ppi(\widetilde{U})$ be an open neighborhood of $0\in \kbb$. By construction, for any $x'\in U$, the reducibility of the fiber $\kxx_{x'}$ is the same as that of the singularity, i.e. the number of the irreducible components $r(\kxx_{x'})$ of $\kxx_{x'}$ is equal to that of the singularities in $\kxx_{x'}\cap \widetilde{U}$. 
As explained in \S\ref{deform XI}, $\kxx_{x'}\cap \widetilde{U}$ has multiple irreducible components only when $x'\in \kbb_{I_{\bullet}}$ for some non-trivial partition $I_{\bullet}\vdash\{1,\cdots,r\}$.
Hence $r(\kxx_{x'})>1$ if and only if $x'\in \kbb_{I_{\bullet}}$.

Let $\delta(\kxx_{x})$ be the $\delta$-invariant of the curve $\kxx_{x}$. In case that $\kbb_{\beta}$ is not one of the strict $\delta$-strata, we have $\dim(V)>\delta(\kxx_{x})$; in case that $\kbb_{\beta}$ is one of the strict $\delta$-strata $\kbb^{\circ}_{I_{\bullet}}$, $I_{\bullet}$ a non-trivial partition of $\{1,\cdots, r\}$, we have the equality $\dim(V)=\delta(\kxx_{x})$, and we have the inequality $r(\kxx_{x})>1$ as explained above. 
So in both case $(1)$ and $(2)$ the dimension of the slice $V$ satisfies the inequality
$$
\dim(V)> \delta(\kxx_{x})+1-r(\kxx_{x}).
$$
Consider the subfamily $\ff^{\xi}_{V}: \overline{\kjj}^{\,\xi}\times_{\kbb}V\to V$ and its composition with the projection $p:V\to \ba^{1}$.
For any $a\in \im(p)$, it is clear that 
$$
\dim(p^{-1}(a))=\dim(V)-1\ge \delta(\kxx_{x})+1-r(\kxx_{x}).
$$ 
By theorem 1.9\footnote{There is a typo in the statement of the theorem, the invariant $\gamma(C)$ in the theorem should be the number of irreducible components of $C$, as indicated in the table of \S2.1.4 of the paper.} of \cite{miglio 2}, the subspace $(p\circ \ff_{V}^{\xi})^{-1}(a)$ is regular along any fiber $(\ff_{V}^{\xi})^{-1}(x'), x'\in p^{-1}(a)$. 
In particular, $p\circ \ff_{V}^{\xi}$ is smooth at the vicinity of $0\in \ba^{1}$.  
The local acyclicity of the smooth morphism then implies that
$$
R\Phi'_{\bar{\eta}_{0}}\bq_{\ell}\big|_{ (\ff^{\xi}_{V})^{-1}(x)}=0,
$$
where $R\Phi'_{\bar{\eta}_{0}}$ is the vanishing cycle functor for the morphism $p\circ \ff^{\xi}_{V}$. Since $\ff^{\xi}_{V}$ is proper, by proper base change theorem, we get 
\begin{equation}\label{xi vanishing 1}
(R\Phi_{\bar{\eta}_{0}}R(\ff^{\xi}_{V})_{*}\bq_{\ell})_{x}=0.
\end{equation}
Notice that the support theorem \ref{cj support} continues to hold for $R(\ff^{\xi}_{V})_{*}\bq_{\ell}$, because $V$ is a transversal slice to the Whitney strata $\kbb_{\beta}$. The equation (\ref{xi vanishing 1}) implies that all the direct summands $(\jj_{!*}\cg^{i})|_{V}$ of $R(\ff^{\xi}_{V})_{*}\bq_{\ell}$ has no vanishing cycle at $x$, i.e.
\begin{equation*}
\big(R\Phi_{\bar{\eta}_0}(\jj_{!*}\cg^{i})\big)_{x}=\big(R^{\dim(V)-1}\Phi_{\bar{\eta}_0}(\jj_{!*}\cg^{i})\big)_{x}=0.
\end{equation*}
This finishes the proof of the cases $(1)$ and $(2)$.
The case $(3)$ is already included as the case $(2)$ of proposition 3.8 of \cite{chen root valuation}, because the fiber $\kxx_{x}$ is irreducible and there is no stability issue for the compactified Jacobian.

\end{proof}

\begin{rem}
For $i=1, \cdots, 2p_{a}(X_{\gamma})-1$, let $\bar{\imath}=\min\{i, 2p_{a}(X_{\gamma})-i\}$, the theorem implies that the critical points on the $\delta$-strata for $\delta>\bar{\imath}$ makes no contribution to the Morse groups of $j_{!*}\cg^{i}$, so
$$
\ch^{*}(j_{!*}\cg^{i})_{0}=H^{*}(F^{\delta_{\gamma}+1-\bar{\imath}}, j_{!*}\cg^{i}). 
$$
It is enough to iterate the above process for $k=\delta_{\gamma},\cdots, \delta_{\gamma}+1-\bar{\imath}$.

\end{rem}

\subsection{The dependence of the cohomologies on the root valuation datum}

We have shown in \cite{chen root valuation} that the cohomologies of the affine Springer fiber $\xx_{\gamma}$ depends only on the root valuation datum of $\gamma$, namely the function 
$$
R:\Phi(G_{\overline{F}}, T_{\overline{F}})\to \bq,\quad R(\alpha)=\val(\alpha(\gamma)) \text{ for } \alpha\in \Phi(G_{\overline{F}}, T_{\overline{F}}).
$$ 
With the same technique, we can show:

\begin{thm}

Both the primitive part of $H^{*}(\xx_{\gamma}, \ql)$ and the cohomologies of the $\xi$-stable quotient $\xx_{\gamma}^{0,\, \xi}/\rS$ depends only on the root valuation datum of $\gamma$. 

\end{thm}

\begin{proof}

For the primitive part of $H^{*}(\xx_{\gamma}, \ql)$, one can either use the microlocal analysis developped in \S\ref{microlocal primitive}, or take the following shortcut: By proposition 5.8 of \cite{chen decomposition}, the primitive part is exactly the $\Lambda^{0}$-invariant subspace of $H^{*}(\xx_{\gamma}^{0}, \ql)$. 
Since $H^{*}(\xx_{\gamma}^{0}, \ql)$ depends only on the root valuation, and the $\Lambda^{0}$-action on it comes from the translation action of $T(F)$ on $\xx_{\gamma}$, the invariant subspace $H^{*}(\xx_{\gamma}^{0}, \ql)^{\Lambda^{0}}$ also depends only on the root valuation datum.  

For the cohomologies of $\xx_{\gamma}^{0, \, \xi}/\rS$, by the equation (\ref{quotient as fiber}), it is enough to show that the fibers $\ch^{*}(\jj_{!*}\cg^{i})_{0}$ depends only on the root valuation datum of $\gamma$. 
Consider the miniversal deformation $\ppi:(\kxx, X_{\gamma})\to (\kbb, 0)$, let $i_{\mu}: \Delta_{\mu}\to \kbb$ be the inclusion of the equisingular strata containing $0$. 
By lemma 4.18 of \cite{chen root valuation}, for semisimple regular elements $\gamma, \gamma'\in \kg[\![\varep]\!]$, if $\gamma$ and $\gamma'$ have the same root valuation datum, then the germs of singularities $\spf(\co[\gamma])$ and $\spf(\co[\gamma'])$ are equisingular. Hence it is enough to show that the restriction of the sheaves $i_{\mu}^{*}\ch^{*}(\jj_{!*}\cg^{i})$ is locally constant over $\Delta_{\mu}$.
For any point $z\in \Delta_{\mu}$, we can perform microlocal analysis of the fiber $(\jj_{!*}\cg^{i})_{z}$ with the same process as explained in \S\ref{micro xi-stable}. 
The Morse groups for this inductive process are calculated in the same way with the same results. 
By theorem \ref{xi morse group}, the Morse groups are non-trivial only over the strict $\delta$-strata $\kbb_{\delta}^{\circ}$, and depend only on the invariant $\delta$. 
By theorem 4.17 of \cite{chen root valuation}, the union of the strict $\delta$-strata $\bigcup_{\delta=0}^{\delta_{\gamma}} \kbb_{\delta}^{\circ}$ is Whitney regular over $\Delta_{\mu}$, 
hence it is locally topologically trivial along $\Delta_{\mu}$ by Thom-Mather's theory of Whitney regularity (cf. \cite{chen root valuation}, \S4.2 for a review).  
Hence the process of patching up the Morse groups to calculate $\ch^{*}(\jj_{!*}\cg^{i})_{z}$ is everywhere the same on $\Delta_{\mu}$, and so the group $\ch^{*}(j_{!*}\cf^{i})_{z}$ must be locally constant on it. This finishes the proof of theorem for $\xx_{\gamma}^{0,\,\xi}/\rS$.

\end{proof}

\bigskip
\small
\noindent
\begin{tabular}{ll}
&Zongbin {\sc Chen} \\ 
\\
&School of mathematics, Shandong University\\
&250100, JiNan, Shandong, \\
&P. R. China \\
&email: {\tt zongbin.chen@email.sdu.edu.cn}

\end{tabular}

\end{document}